\documentclass[twoside]{aiml22}

\usepackage{aiml22macro}

\usepackage{graphicx}
\usepackage{amsmath}
\usepackage{amssymb}
\usepackage{bussproofs} 
\usepackage{proof} 
\usepackage{float}
\usepackage{euscript}

\newcommand{\imp}{\rightarrow}

\newcommand{\en}{\wedge} 
\newcommand{\of}{\vee} 
\newcommand{\ifff}{\leftrightarrow}
\newcommand{\E}{\exists}
\newcommand{\A}{\forall} 
\newcommand{\bx}{\Box}
\newcommand{\af}{\vdash}
\newcommand{\sml}{\ll}

\newcommand{\seq}{\Rightarrow}
\newcommand{\rsch}{{\EuScript R}}
\newcommand{\cald}{{\EuScript D}}

\renewcommand{\phi}{\varphi}
\newcommand{\gam}{\gamma}
\newcommand{\Ga}{\Gamma}
\newcommand{\De}{\Delta}

\newcommand{\GthW}{\mathbf{G4w}}

\newcommand{\Ap}{\forall^{\circ}\hspace{-.3mm}p\hspace{.2mm}}
\newcommand{\Apd}{\forall^{\diamond}\hspace{-.3mm}p\hspace{.2mm}}
\newcommand{\Ep}{\exists^{\circ}\hspace{-.3mm}p\hspace{.2mm}}
\newcommand{\Epd}{\exists^{\diamond}\hspace{-.3mm}p\hspace{.2mm}}
\newcommand{\Apm}{\forall^{\circ}\!\!\!_{{m}} \hspace{-.08mm}p\hspace{.2mm}} 
\newcommand{\Epm}{\exists^{\circ}\!\!\!_{{m}} \hspace{-.08mm}p\hspace{.2mm}}

\newcommand{\Apax}{\forall^{\circ}\!\!\!_{{ax}} \hspace{-.08mm}p\hspace{.2mm}}
\newcommand{\Apat}{\forall^{\circ}\!\!\!_{{at}} \hspace{-.08mm}p\hspace{.2mm}}
\newcommand{\Epax}{\exists^{\circ}\!\!\!_{ax} \hspace{-.08mm}p\hspace{.2mm}}
\newcommand{\Epdax}{\exists^{\diamond}\!\!\!_{ax} \hspace{-.08mm}p\hspace{.2mm}}
\newcommand{\Epat}{\exists^{\circ}\!\!\!_{at} \hspace{-.08mm}p\hspace{.2mm}}

\def\lastname{
Akbar Tabatabai, Iemhoff and Jalali}

\begin{document}

\begin{frontmatter}
  \title{Uniform Lyndon interpolation for intuitionistic monotone modal logic}
  \author{Amirhossein Akbar Tabatabai}
  \author{Rosalie Iemhoff}
  \author{Raheleh Jalali}\footnote{Support by the Netherlands Organisation for Scientific Research under grant 639.073.807 and by the FWF project P 33548 is gratefully acknowledged.}
  \address{University of Groningen \\ Bernoulliborg,  Nijenborgh 9 \\ 9747 AG Groningen, the Netherlands}
  \address{Utrecht University \\ Janskerkhof 13 \\ 3512 BL Utrecht, the Netherlands}

  \begin{abstract}
In this paper we show that the intuitionistic monotone modal logic $\mathsf{iM}$ has the uniform Lyndon interpolation property (ULIP). The logic $\mathsf{iM}$ is a non-normal modal logic on an intuitionistic basis, and the property ULIP is a strengthening of interpolation in which the interpolant depends only on the premise or the conclusion of an implication, respecting the polarities of the propositional variables. Our method to prove ULIP yields explicit uniform interpolants and makes use of a terminating sequent calculus for $\mathsf{iM}$ that we have developed for this purpose. As far as we know, the results that $\mathsf{iM}$ has ULIP and a terminating sequent calculus are the first of their kind for an intuitionistic non-normal modal logic.
However, rather than proving these particular results, our aim is to show the flexibility of the constructive proof-theoretic method that we use for proving ULIP. It has been developed over the last few years and has been applied to substructural, intermediate, classical (non-)normal modal and intuitionistic normal modal logics. In light of these results, intuitionistic non-normal modal logics seem a natural next class to try to apply the method to, and we take the first step in that direction in this paper. 
  \end{abstract}

 \end{frontmatter}

\section{Introduction}
Over the last years a method to prove uniform (Lyndon) interpolation has been developed by the authors that applies to various (intuitionistic) modal and intermediate logics \cite{Iemhoff,Iemhoffa,Iemhoffc,Iemhoffetal,Jalali&Tabatabai,Jalali&Tabatabaia}. Uniform interpolation is a strengthening of interpolation in which the interpolant depends only on the premise or the conclusion of an implication. It is Lyndon whenever the interpolant in addition respects the polarities of the propositional variables involved. Our method to prove the property is based on sequent calculi for the given logics. Until now, it has been applied to classical modal logics, normal as well as non-normal, but in the intuitionistic setting only to intermediate logics and to intuitionistic modal logics that are normal. 

In this paper, we try to show the general applicability of the method by applying it to a well-known intuitionistic non-normal modal logic namely the intuitionistic monotone modal logic {\sf iM}, which is axiomatized over intuitionistic propositional logic $\mathsf{IPC}$ by the following axiom and rule \cite{Olivetti}: 
\begin{center}
    $\bx (\phi \en \psi) \imp \bx\phi \en \bx \psi$ 
    \ \ \ \ 
    \AxiomC{$\phi \imp \psi$}
    \AxiomC{$\psi\imp \phi$}
    \RightLabel{$E$}
    \BinaryInfC{$\bx\phi \imp\bx\psi$}
    \DisplayProof
\end{center}
The axiom is one direction of the principle $\bx(\phi \en \psi) \ifff \bx\phi\en\bx\psi$ that holds in every normal modal logic. We show that $\mathsf{iM}$ 
has uniform Lyndon interpolation, which, to our knowledge, is the first result of this kind, meaning the first result stating that an intuitionistic non-normal modal logic has uniform (Lyndon) interpolation. Our method is effective in that it provides explicit (existential and universal) interpolants and it makes use of a terminating sequent calculus for the logic. The calculus is an extension of the calculus $\mathbf{G4ip}$, which has been introduced by Dyckhoff as a variant of $\mathbf{G3ip}$ in which proof search terminates (without extra conditions on the search) \cite{Dyckhoff}. The terminating calculus that we develop here seems to be the first terminating calculus for the logic {\sf iM}. 
Our method  to prove uniform interpolation is inspired by the first syntactic proof of uniform interpolation, given by Pitts for intuitionistic propositional logic \cite{Pitts}.

As can be seen from \cite{Olivetti}, the semantics for intuitionistic non-normal modal logic that combines the semantics of intuitionistic logic and classical non-normal modal logic is not simple. In this light it is somewhat surprising that the proof-theoretic method developed in this paper is essentially not more complicated than the one for its normal counterpart, that, we have to admit, is already quite complicated in itself.  

In the literature there are many syntactic proofs of Craig interpolation, most of them connected in some way or another to the well-known syntactic {\em Maehara method} \cite{Maehara}.  Proofs of uniform interpolation are less common, and syntactic proofs of uniform interpolation even more so. Most of the existing proofs are inspired by Pitts' syntactic proof of uniform interpolation for $\mathsf{IPC}$ \cite{Pitts} mentioned earlier. Some proof systems seem to lend themselves better for syntactic proofs of (uniform) interpolation than others. Especially for nested sequents there are several syntactic results. There are nested sequent systems for certain tense logics and bi-intuitionistic logic that allow for a syntactic proof of Craig interpolation \cite{Lyon}. A similar statement holds for various modal and intermediate logics, although in this case the method is no longer purely syntactical but also contains semantic elements \cite{Fitting,Kuznets}.

Uniform interpolation has applications in computer science, in particular in description logics \cite{Lutzetal}, but our interest in the property stems from a project in universal proof theory where we aim to develop methods to prove that certain (classes of) logics cannot have certain well-behaved proof systems, in our case sequent calculi \cite{Jalali&Tabatabaia,Iemhoffa}. Here we make use of the fact that uniform interpolation seems to be a rare property among logics. For example, only seven intermediate logics have this property \cite{Ghilardi,Maksimova}. In fact, in this case also only seven intermediate logics have Craig interpolation, but in modal logics these two properties in general do not coincide. And while in this class the property is equivalent to interpolation, this is certainly not the case for modal logics. The logics $\mathsf{K4}$ and $\mathsf{S4}$ are examples of logics that have Craig interpolation but not uniform interpolation \cite{Bilkova,Ghilardi}. The result in this paper is meant as a first step to also consider the class of intuitionistic non-normal modal logics in the project. 
The reason that we only treat one logic in that class and only prove that it has uniform Lyndon interpolation without taking the further generalization steps needed for the project, is more for reasons of space than anything else. We hope to take these further steps and cover more intuitionistic non-normal modal logics in the future.

This paper is built up as follows. Section~\ref{sectionpreliminaries} contains the preliminaries, in which the interpolation properties, the intuitionistic non-normal modal logic $\mathsf{iM}$, and the sequent calculi $\mathbf{G3iM^w}$ and $\mathbf{G4w}$ are defined. In Section~\ref{SecTerminatingProofSystems}, the terminating calculus $\mathbf{G4iM}$ is introduced and is shown to be equivalent to $\mathbf{G3iM^w}$, which implies that it is a terminating calculus for $\mathsf{iM}$. In Section~\ref{sectionuniform}, it is proved that $\mathsf{iM}$ has uniform Lyndon interpolation property.

\section{Preliminaries}
 \label{sectionpreliminaries}
The language we use is $\mathcal{L}=\{\wedge, \vee, \to, \Box, \bot\}$, and $\top$ is an abbreviation for $\bot \to \bot$, as usual. We use small Roman letters $p, q, \ldots$ for atomic formulas, small Greek letters $\phi, \psi, \ldots$ to denote formulas, and capital Greek letters $\Sigma, \Delta, \ldots$ and also $\bar{\phi}, \bar{\psi}, \ldots$ to denote finite multisets of formulas.
The \emph{weight} of a formula is defined as follows, which is a combination of the definitions given in \cite{Bilkova} and \cite{Dyckhoff}: $w(p)=w(\bot)=w(\top)=1$, for any atomic formula $p$, $w(\phi \odot \psi)=w(\phi)+ w(\psi)+1$, for $\odot \in \{\vee, \to\}, w(\phi \wedge \psi)=w(\phi)+ w(\psi)+2,$ and $w(\Box \phi)=w(\phi)+1$. This weight function induces an ordering on 
the multisets: $\Gamma \prec \Delta$ if $\Gamma$ is obtained from $\Delta$ by replacing one or more formulas of $\Delta$ by zero or more formulas, each of which is of a strictly lower weight. Note that this order is well-founded.
\begin{definition}
The sets of positive and negative variables of a formula $\phi \in \mathcal{L}$, denoted by $V^+(\phi)$ and $V^-(\phi)$, respectively, are defined recursively by:

\begin{itemize}
\item[$\bullet$]
$V^+(p)=\{p\}$, \hspace{-3pt}$V^-(p)=
V^+(\top)=V^-(\top)=V^+(\bot)=V^-(\bot)=\varnothing$ \hspace{-4pt} for \hspace{-4pt} 
atom $p$,
\item[$\bullet$]

$V^+(\phi \odot \psi)=V^+(\phi) \cup V^+(\psi)$, 
$V^-(\phi \odot \psi)=V^-(\phi) \cup V^-(\psi)$, for 
$\odot \in \{\wedge, \vee\}$, 
\item[$\bullet$]

$V^+(\phi \to \psi)=V^-(\phi) \cup V^+(\psi)$ and $V^-(\phi \to \psi)=V^+(\phi) \cup V^-(\psi)$,
\item[$\bullet$]

$V^+(\Box \phi)=V^+(\phi)$  and $V^-(\Box \phi)=V^-(\phi)$.
\end{itemize}
\noindent Define $V(\phi)$ as $V^+(\phi) \cup V^-(\phi)$ and set $V^+(\Gamma)=\bigcup_{\gamma \in \Gamma} V^+(\gamma)$ and $V^-(\Gamma)=\bigcup_{\gamma \in \Gamma} V^-(\gamma)$, for a multiset $\Gamma$. For an atomic formula $p$, a formula $\phi$ is called {\em $p^+$-free} ({\em $p^-$-free}), if $p \notin V^+(\phi)$ ($p \notin V^-(\phi)$). It is called {\em $p$-free} if $p \notin V(\phi)$. Note that a formula is $p$-free iff $p$ does not occur anywhere in the formula.
\end{definition}

\noindent We will need the following notations: if we want to refer to both $V^+(\phi)$ and $V^-(\phi)$, we use $V^{\dagger}(\phi)$, with the condition ``for any $\dagger \in \{+, -\}$". When we want to refer to both $V^+(\phi)$ and $V^-(\phi)$ but also to their duals,
we use the notation $V^{\circ}(\phi)$ for the one we intend and $V^{\diamond}(\phi)$\footnote{The superscript $\diamond$ has nothing to do with the usual modality $\Diamond$.} for its dual, respectively.
Therefore, by the sentence `if $p \in V^{\circ}(\phi)$, then $p \notin V^{\diamond}(\phi)$, for any $\circ, \diamond \in \{+, -\}$', we mean `if $p \in V^{+}(\phi)$, then $p \notin V^{-}(\phi)$ and if $p \in V^{-}(\phi)$, then $p \notin V^{+}(\phi)$'. 

\begin{definition} \label{Logic}
A \emph{logic} $L$ is a set of formulas of $\mathcal{L}$ extending the set of intuitionistic tautologies, $\mathsf{IPC}$, and closed under substitution and modus ponens $\phi, \phi \to \psi \vdash \psi$.
\end{definition}

\begin{definition} \label{DfnLyndonInterpolation}
A logic $L$ has the \emph{Lyndon interpolation property (LIP)} if for any formulas $\phi, \psi \in \mathcal{L}$ such that $L \vdash \phi \to \psi$, there is a formula $\theta \in \mathcal{L}$ such that $V^{\dagger}(\theta) \subseteq V^{\dagger}(\phi) \cap V^{\dagger}(\psi)$, for any $\dagger \in \{+, -\}$ and $L \vdash \phi \to \theta$ and $L \vdash \theta \to \psi$. A logic has \emph{Craig interpolation (CIP)} if it has 
the above properties, omitting all the superscripts $\dagger \in \{+, -\}$. 
\end{definition}

\begin{definition} \label{DfnUniformInterpolation} 
A logic $L$ has the \emph{uniform Lyndon interpolation property (ULIP)} if for any formula $\phi \in \mathcal{L}$, 
atom 
$p$, and 
$\circ \in \{+, -\}$, there are 
$p^{\circ}$-free formulas, 
$\Ep \phi$ and $\Ap \phi$, such that $V^{\dagger}(\Ep \phi) \subseteq V^{\dagger}(\phi)$ and $V^{\dagger}(\Ap \phi) \subseteq V^{\dagger}(\phi)$, for any 
$\dagger \in \{+, -\}$ and
\begin{description}
\item[$(i)$]
$L \vdash \phi \to \Ep \phi$, and
\item[$(ii)$]
for any $p^{\circ}$-free formula $\psi$ if $L \vdash \phi \to \psi$ then $L \vdash \Ep \phi \to \psi$, 
\item[$(iii)$]
$L \vdash \Ap \phi \to \phi$,
\item[$(iv)$]
for any $p^{\circ}$-free formula $\psi$ if $L \vdash \psi \to \phi$ then $L \vdash \psi \to \Ap \phi$,
\end{description}
A logic has \emph{uniform interpolation property (UIP)} if it has all the above properties, omitting the superscripts $\circ, \dagger \in \{+, -\}$, everywhere. Note that although the interpolants are indicated by expressions that contain symbols that do not belong to $\mathcal{L}$, they do stand for formulas in the language $\mathcal{L}$. 
\end{definition}

\begin{theorem}\label{ThmULIPresultsUIP}
If a logic $L$ has ULIP, 
then it has both LIP and UIP.
\end{theorem}
\begin{proof}
For UIP, define $\forall p \phi=\forall^+ p \forall^- p \phi$ and $\exists p \phi=\exists^+ p \exists^- p \phi$. We will show that $\forall p \phi$ acts as the uniform interpolant for $\phi$. The case for $\exists p \phi$ is similar. By definition, $V^{\dagger}(\forall^+ p \forall^- p \phi) \subseteq V^{\dagger}(\forall^- p \phi) \subseteq V^{\dagger}(\phi)$, for any $\dagger \in \{+,-\}$. Therefore, $V(\forall p \phi) \subseteq V(\phi)$. Moreover, $\forall^+ p \forall^- p \phi$ is $p^+$-free, by definition. Suppose $p \in V^-(\forall^+ p \forall^- p \phi)$. Then, $p \in V^-(\forall^- p \phi)$, which is a contradiction, as $\forall^- p \phi$ is $p^-$-free. Hence, $\forall p \phi$ is $p$-free. 

Condition $(iii)$ is easy, as we have $L \vdash \forall^+ p \forall^- p \phi \to  \forall^- p \phi$ and $L \vdash  \forall^- p \phi \to \phi$, by Definition \ref{DfnUniformInterpolation}. Therefore, $L \vdash \forall p \phi \to \phi$. For condition $(iv)$, let $\psi$ be a $p$-free formula such that $L \vdash \psi \to \phi$. Then, as $\psi$ is $p^-$-free, we have $L \vdash \psi \to \forall^- p \phi$ and as $\psi$ is $p^+$-free, we get $L \vdash \psi \to \forall^+ p \forall^- p \phi$.

For LIP, 
assume $L \vdash \phi \to \psi$. Define
$\theta=\exists^+ P^+ \exists^- P^- \phi$, where $P^{\dagger}=V^{\dagger}(\phi)-[V^{\dagger}(\phi) \cap V^{\dagger}(\psi)]$, for any $\dagger \in \{+, -\}$ and by $\exists^{\dagger} \{p_1, \dots, p_n\}^{\dagger}$ we mean $\E p^{\dagger}_1 \dots \E p^{\dagger}_n$. Since $\theta$ is $p^{\dagger}$-free for any $p \in P^{\dagger}$ and any $\dagger \in \{+, -\}$, we have $V^{\dagger}(\theta) \subseteq V^{\dagger}(\phi)-P^{\dagger} \subseteq V^{\dagger}(\phi) \cap V^{\dagger}(\psi)$. 
For the other condition, it is clear that $L \vdash \phi \to \theta$ and as $\psi$ is $p^{\dagger}$-free, for any $p \in P^{\dagger}$, we have $L \vdash \theta \to \psi$.
\end{proof}

\subsection{Sequent calculi}\label{sectionsequent}
A \emph{sequent} $S$ is any expression of the form $\Gamma \Rightarrow \Delta$, where $\Gamma$ and $\Delta$ are two multisets of formulas called the \emph{antecedent} and the \emph{succedent} of the sequent, denoted by $S^{a}$ and $S^{s}$, respectively. A sequent is called \emph{single-conclusion} if its succedent has at most one formula. The \emph{multiplication} of the sequents $S$ and $T$ is defined by $S \cdot T = (S^{a} \cup T^{a})\Rightarrow (S^{s} \cup T^{s})$. The set of {\em positive variables} $(V^+)$ and the set of {\em negative variables} $(V^-)$ of a sequent $S$ are defined by $V^{\circ} (S)=V^{\diamond} (S^{a}) \cup V^{\circ} (S^{s})$, for any $\circ \in \{+, -\}$ and $V(S)= V(S^{a}) \cup V(S^{s})$. In case it is clear from the notation which set we mean we omit the words ``positive'' and ``negative''. The ordering $\prec$ can be extended to sequents by $S \prec T:= S^{a} \cup S^s \prec T^{a} \cup T^s$. If $S \prec T$, we say that $S$ is \emph{lower} than $T$. Sequents and multisets can also be compared with each other in an expected way. For instance, $\Sigma \prec S$ means $\Sigma \prec (S^{a} \cup S^s)$. A \emph{rule} is an expression of the form 
   \begin{tabular}{c}
         \AxiomC{$S_1 \; \ldots \; S_n$}
 \UnaryInfC{$S$}
 \DisplayProof
    \end{tabular}
where $S_1, \ldots, S_n,$ and $S$ are sequents called the \emph{premises} and the \emph{conclusion} of the rule, respectively.
If a sequent $S$ is the conclusion of an instance of a rule, we say that the rule is \emph{backward applicable} to $S$.  A \emph{sequent calculus} is a set of rules. In this paper, we consider single-conclusion sequent calculi, where only single-conclusion sequents are allowed.

\noindent Let us introduce three sequent calculi that we need throughout the paper. The first is the sequent calculus $\mathbf{G3iM^w}$\footnote{The use of superscript $\mathbf{^w}$ becomes clear in the next section.} presented in Figure \ref{figG3iM}. The system was introduced in \cite{Olivetti} under the name $\mathsf{G}.\Box\!-\!\mathsf{IM}$.
\begin{figure}
\begin{center}
 \begin{tabular}{ccc}
 $\Gamma, p \Rightarrow p \qquad Ax$
 &
$\Gamma, \bot \Rightarrow \phi \qquad L \bot$
 \\[1.5ex]  
 \AxiomC{$\Gamma, \phi, \psi \Rightarrow \theta$}
 \RightLabel{\small$L \wedge$} 
 \UnaryInfC{$\Gamma, \phi \wedge \psi \Rightarrow \theta$}
 \DisplayProof
 &
 \AxiomC{$\Gamma \Rightarrow \phi$}
 \AxiomC{$\Gamma \Rightarrow \psi$}
 \RightLabel{\small $R \wedge$} 
 \BinaryInfC{$\Gamma \Rightarrow \phi \wedge \psi$}
 \DisplayProof
  \\[3ex]
 \AxiomC{$\Gamma, \phi \Rightarrow \theta$}
 \AxiomC{$\Gamma, \psi \Rightarrow \theta$}
 \RightLabel{\small$L \vee$} 
 \BinaryInfC{$\Gamma, \phi \vee \psi \Rightarrow \theta$}
 \DisplayProof
 &
 \AxiomC{$\Gamma \Rightarrow \phi_i$}
 \RightLabel{\small$R \vee$ ($i=0,1$)} 
 \UnaryInfC{$\Gamma \Rightarrow \phi_0 \vee \phi_1$}
 \DisplayProof 
 \\[3ex]
 \AxiomC{$\Gamma, \phi \to \psi \Rightarrow \phi$}
 \AxiomC{$\Gamma, \psi \Rightarrow \theta$}
 \RightLabel{\small$L\! \to$} 
 \BinaryInfC{$\Gamma, \phi \to \psi \Rightarrow \theta$}
 \DisplayProof
 &
 \AxiomC{$\Gamma, \phi \Rightarrow \psi$}
 \RightLabel{\small$R\! \to$} 
 \UnaryInfC{$\Gamma \Rightarrow \phi \to \psi$}
 \DisplayProof \\
\end{tabular}
\begin{tabular}{c}
\hspace{20pt}
  \AxiomC{$\phi \Rightarrow \psi$}
 \RightLabel{\small$M^{seq}_{\Box}$} 
 \UnaryInfC{$\Gamma, \Box \phi \Rightarrow \Box \psi$}
  \DisplayProof
\end{tabular}
\caption{The sequent calculus $\mathbf{G3iM^w}$. In $Ax$, $p$ must be an atom.}
\label{figG3iM}
\end{center}
\end{figure}
The logic of $\mathbf{G3iM^w}$, i.e., the set of all formulas such that $(\Rightarrow \phi)$ is derivable in $\mathbf{G3iM^w}$, is called $\mathsf{iM}$, the intuitionistic monotone modal logic. The second system is $\mathbf{G4ip}$, a sequent calculus for $\mathsf{IPC}$ presented in Figure \ref{figg4ip} and introduced in \cite{Dyckhoff}.
\begin{figure}
\begin{center}
 \begin{tabular}{ccc}
 $\Gamma, p \Rightarrow p \qquad Ax$
 &
$\Gamma, \bot \Rightarrow \phi \qquad L \bot$
 \\[1.5ex]
 \AxiomC{$\Gamma, \phi, \psi \Rightarrow \Delta$}
 \RightLabel{\small$L\wedge$} 
 \UnaryInfC{$\Gamma, \phi \wedge \psi \Rightarrow \Delta$}
 \DisplayProof
 &
 \AxiomC{$\Gamma \Rightarrow \phi$}
 \AxiomC{$\Gamma \Rightarrow \psi$}
 \RightLabel{\small$R \wedge$} 
 \BinaryInfC{$\Gamma \Rightarrow \phi \wedge \psi$}
 \DisplayProof
  \\[3ex]
 \AxiomC{$\Gamma, \phi \Rightarrow \Delta$}
 \AxiomC{$\Gamma, \psi \Rightarrow \Delta$}
 \RightLabel{\small$L \vee$} 
 \BinaryInfC{$\Gamma, \phi \vee \psi \Rightarrow \Delta$}
 \DisplayProof
 &
 \AxiomC{$\Gamma \Rightarrow \phi_i$}
 \RightLabel{\small$R\vee_i (i=0, 1)$} 
 \UnaryInfC{$\Gamma \Rightarrow \phi_0 \vee \phi_1$}
 \DisplayProof
 \\[3ex]
 \AxiomC{$\Gamma, p, \psi \Rightarrow \Delta$}
 \RightLabel{\small$Lp \!\to$} 
 \UnaryInfC{$\Gamma, p, p \to \psi \Rightarrow \Delta$}
 \DisplayProof
 &
 \AxiomC{$\Gamma, \phi_1 \to (\phi_2 \to \psi) \Rightarrow \Delta$}
 \RightLabel{\small$L\wedge \!\to$} 
 \UnaryInfC{$\Gamma, (\phi_1 \wedge \phi_2) \to \psi \Rightarrow \Delta$}
 \DisplayProof
 \\[3ex]
 \AxiomC{$\Gamma, \phi_1 \to \psi, \phi_2 \to \psi \Rightarrow \Delta$}
 \RightLabel{\small$L \vee \!\to$} 
 \UnaryInfC{$\Gamma, \phi_1 \vee \phi_2 \to \psi \Rightarrow \Delta$}
 \DisplayProof
 &
 \AxiomC{$\Gamma, \phi \Rightarrow \psi$}
 \RightLabel{\small$R\! \to$} 
 \UnaryInfC{$\Gamma \Rightarrow \phi \to \psi$}
 \DisplayProof
 \\[3ex]
\multicolumn{2}{c}{ \AxiomC{$\Gamma, \phi_2 \to \psi \Rightarrow \phi_1 \to \phi_2$}
  \AxiomC{$\Gamma, \psi \Rightarrow \Delta$}
 \RightLabel{\small$L\!\to \to$} 
 \BinaryInfC{$\Gamma, (\phi_1 \to \phi_2) \to \psi \Rightarrow \Delta$}
 \DisplayProof}
\end{tabular}
\end{center}
\caption{The sequent calculus $\mathbf{G4ip}$ from \cite{Dyckhoff}. In $Ax$ and $(Lp\!\to)$, $p$ is an atom.}
\label{figg4ip}
\end{figure}
\noindent If we add the following weakening rules to $\mathbf{G4ip}$, we get the third system $\mathbf{G4w}$:
\begin{center}
    \begin{tabular}{c c}
         \AxiomC{$\Gamma \Rightarrow \Delta$}
 \RightLabel{\small$L w$} 
 \UnaryInfC{$\Gamma, \phi \Rightarrow \Delta$}
 \DisplayProof
 &
      \AxiomC{$\Gamma \Rightarrow $}
 \RightLabel{\small$R w$} 
 \UnaryInfC{$\Gamma \Rightarrow \phi$}
 \DisplayProof
    \end{tabular}
\end{center}
The weakening rules are admissible in $\mathbf{G4ip}$ and hence there is no need to include them explicitly. However, as we will work with an extension of the system $\mathbf{G4ip}$, we will need the explicit weakening rules later.

As the final part of this section, let us mention some of the properties of the rules in $\mathbf{G4w}$ that we will need later. First, note that in any of the rules of $\mathbf{G4w}$, $\Gamma$ and $\Delta$ are free for any multiset substitution. We call this property the \emph{free-context property}. For later reference, we call any premise of a rule with $\Delta$ in its succedent \emph{contextual} and the other premises \emph{non-contextual}. Second, if we denote the set of rules in $\mathbf{G4w}$ minus the rule $(Lp\!\to)$ by $\mathbf{G4w}^-$, then all the rules of $\mathbf{G4w}^-$, have one of the following general forms:
\begin{center}
 \begin{tabular}{c c}
 \AxiomC{$\{ \Gamma, \bar{\phi}_i \Rightarrow {\bar{\delta}}_i \}_{i \in I}$}
 \AxiomC{$\{ \Gamma, \bar{\psi}_j \Rightarrow \Delta \}_{j \in J}$}
 \BinaryInfC{$\Gamma, \phi \Rightarrow \Delta$}
 \DisplayProof \;\;\;
 &
 \AxiomC{$\{ \Gamma, \bar{\phi}_i \Rightarrow {\bar{\delta}}_i \}_{i \in I}$}
 \UnaryInfC{$\Gamma \Rightarrow \phi$}
 \DisplayProof
\end{tabular}
\end{center}
where $I$ and $J$ are some finite (possibly empty) sets, $\Gamma$ and $\Delta$ are free for any multiset substitution and $\bar{\phi}_i$'s, $\bar{\psi}_i$'s and  $\bar{\delta}_i$'s are (possibly empty) multisets of formulas, where $\bar{\delta}_i$'s are either empty or a singleton. The formula $\phi$ is called the \emph{main} formula and the formulas in $\bar{\phi}_i, \bar{\psi}_j,$ and $\bar{\delta}_i$ are called the \emph{active} formulas of the rule. If the main formula is in the antecedent (succedent), the rule is called a \emph{left} (\emph{right}) rule.  Third, notice that each rule in $\mathbf{G4w}^-$ enjoys the \emph{local variable preserving property}, i.e., given $\circ \in \{+, -\}$, for the left rule, we have $\bigcup_i \bigcup_{\theta \in \bar{\phi}_i} V^{\circ}(\theta) \cup \bigcup_j \bigcup_{\theta \in \bar{\psi}_j} V^{\circ}(\theta) \cup \bigcup_i \bigcup_{\theta \in \bar{\delta}_i} V^{\diamond}(\theta) \subseteq V^{\circ}(\phi),$ and for the right one,
$
\bigcup_i \bigcup_{\theta \in \bar{\phi}_i} V^{\diamond}(\theta) \cup \bigcup_i \bigcup_{\theta \in \bar{\delta}_i} V^{\circ}(\theta) \subseteq V^{\circ}(\phi).$ 
This property ensures the crucial condition $\bigcup_{i=1}^n V^{\circ}(S_i) \subseteq V^{\circ}(S)$, for any instance of the rule
\begin{tabular}{c c c}
 \AxiomC{$S_1 \; \cdots \; S_n$}  
 \UnaryInfC{$S$}
 \DisplayProof
\end{tabular}
in $\mathbf{G4w}^-$ and any $\circ \in \{+, -\}$. We call this weaker property the \emph{variable preserving property}. 
Note that the rule $(Lp\!\to)$ also enjoys this property. Finally, notice that in any rule in $\mathbf{G4w}$, each of the premises is lower than the conclusion in the order $\prec$. 

\section{A Terminating sequent calculus} \label{SecTerminatingProofSystems}
In this section, we provide a terminating single-conclusion sequent calculus for $\mathsf{iM}$. \hspace{-6pt} Define the system $\mathbf{G4iM}$ as $\mathbf{G4w}$ extended by the following rules:
\begin{center}
\begin{tabular}{c c c}
 \AxiomC{$\phi \Rightarrow \psi$}
 \RightLabel{\small$M$} 
 \UnaryInfC{$\Box \phi \Rightarrow \Box \psi$}
 \DisplayProof
&
 \AxiomC{$ \phi \Rightarrow \psi$} 
  \AxiomC{$\Gamma, \Box \phi, \theta \Rightarrow \Delta$}
   \RightLabel{\small$LM\!\!\to$} 
 \BinaryInfC{$\Gamma, \Box \phi, \Box \psi \to \theta \Rightarrow \Delta$}
 \DisplayProof
\end{tabular}
\end{center}
Note that each premise of any rule in $\mathbf{G4iM}$ is lower than the conclusion. Consequently, $\mathbf{G4iM}$ is terminating, i.e., any proof search terminates. 

The remainder of this section is devoted to the proof that $\mathbf{G3iM^w}$ and $\mathbf{G4iM}$ are equivalent, the main part of which consists of a proof that $\mathbf{G3iM^w}$ and $\mathbf{G4iM^w}$ are equivalent. This proof is an adaptation of a similar result for $\mathbf{G3i}$ and $\mathbf{G4i}$ in \cite{Dyckhoff}. We start with some preliminaries. 

\subsection{Strict proofs in $\mathbf{G3iM^w}$}
\begin{lemma}
 \label{lemimpconv}
All rules in $\mathbf{G3iM^w}$ except $R\of$ and $L\!\imp$ are invertible, and $\mathbf{G3iM^w}$ 
is closed under weakening, contraction and {\em implication inversion}, i.e.\ the following rule is admissible:
\[
 \AxiomC{$\Ga,\phi\imp\psi \seq \De$}
 \UnaryInfC{$\Ga, \psi \seq \De$}
 \DisplayProof
\]
\end{lemma}
\begin{proof}
Closures under the structural rules and implication inversion are proved with induction to the depth of the derivation.
\end{proof}

A multiset is {\em irreducible} if it has no element that is a disjunction or a conjunction or falsum and for no atom $p$ does it contain both $p\imp\psi$ and $p$. A sequent $S$ is {\em irreducible} if $S^a$ is. A proof is {\em sensible} if its  last inference does not have a principal formula on the left of the form $p\imp\psi$ for some atom $p$ and formula $\psi$.\footnote{In \cite{Iemhoffb} the requirement that the principal formula be on the left was erroneously omitted.} A proof in $\mathbf{G3iM^w}$ is {\em strict} if in the last inference, in case it is an instance of $L\!\imp$ with principal formula $\bx\phi \imp \psi$, the left premise is an axiom or the conclusion of an application of the modal rule. 

\begin{lemma} 
 \label{lemstrict} 
Every irreducible sequent that is provable in $\mathbf{G3iM^w}$ has a sensible strict proof in $\mathbf{G3iM^w}$. 
\end{lemma}
\begin{proof}
This is proved in the same way as the corresponding lemma (Lemma 1) in \cite{Dyckhoff}. Arguing by contradiction, assume that among all provable irreducible sequents that have no sensible strict proofs, $S$ is such a sequent with the shortest proof, $\cald$, where the {\em length} of a proof is the length of its leftmost branch. Thus the last inference in the proof is an application 
\[
 \infer{\Ga,\phi\imp \psi \seq \De}{
 \deduce[\cald_1]{\Ga, \phi\imp \psi \seq \phi}{} & 
 \deduce[\cald_2]{\Ga, \psi \seq \De}{} }
\]
of L$\imp$, where $\phi$ is an atom or a modal formula.  
Since $S^a$ is irreducible, $\bot \not\in S^a$ and if $\phi$ is an atom, $\phi\not\in S^a$. Therefore the left premise cannot be an instance of an axiom and hence is the conclusion of a rule, say $\rsch$. Since the succedent of the conclusion of $\rsch$ consists of an atom or a modal formula, $\rsch$ is a left rule or a right modal rule. The latter case cannot occur, since the proof then would be strict and sensible. Thus $\rsch$ is a left rule. 

We proceed as in \cite{Dyckhoff}. Sequent $(\Ga, \phi\imp \psi \seq \phi)$ is irreducible and has a shorter proof than $S$. Thus its subproof $\cald_1$ is strict and sensible. Since the sequent is irreducible and $\phi$ is an atom or a pure modal formula, the last inference of $\cald_1$ is L$\imp$ with a principal formula $\phi'\imp \psi'$ such that $\phi'$ is not an atom. Let $\cald'$ be the proof of the left premise $(\Ga, \phi\imp \psi \seq \phi')$. Thus the last part of $\cald$ looks as follows, where $\Pi,\phi'\imp \psi'=\Ga$.  
\[
 \infer{\Pi, \phi\imp \psi,\phi'\imp \psi' \seq \De}{
 \infer{\Pi, \phi\imp \psi,\phi'\imp \psi' \seq \phi}{
  \deduce[\cald']{\Pi, \phi\imp \psi,\phi'\imp \psi' \seq \phi'}{} & 
  \deduce[\cald'']{\Pi, \phi\imp \psi,\psi' \seq \phi}{}} & 
 \deduce[\cald_2]{\Pi, \psi,\phi'\imp \psi' \seq \De}{} }
\]
Consider the following proof of $S$.  
\[
 \infer{\Pi, \phi\imp \psi,\phi'\imp \psi' \seq \De}{
  \deduce[\cald']{\Pi, \phi\imp \psi,\phi'\imp \psi' \seq \phi'}{} & 
  \infer{\Pi, \phi\imp \psi,\psi'\seq \De}{ 
   \deduce[\cald'']{\Pi, \phi\imp \psi,\psi' \seq \phi}{} & 
   \deduce[\cald''']{\Pi,\psi,\psi' \seq \De}{} 
  } 
 }
\]
The existence of $\cald'''$ follows from Lemma~\ref{lemimpconv} and the existence of $\cald_2$. The obtained proof is strict and sensible: In case $\phi'$ is not a modal formula, this is straightforward. In case $\phi'$ is a modal formula, it follows from the fact that was observed above, namely that $\cald_1$ is strict and sensible. 
\end{proof}

\begin{theorem} 
 \label{thmequivalence}
$\mathbf{G3iM^w}$ and $\mathbf{G4iM^w}$ are equivalent (derive exactly the same sequents). 
\end{theorem}
\begin{proof}
The proof is an adaptation of the proof of Theorem 3.4 in \cite{Iemhoffb}, which again is an adaptation of Theorem 1 in \cite{Dyckhoff}. Under the assumptions in the theorem we have to show that for all sequents $S$: 
$\af_{\mathbf{G3iM^w}} S$ if and only if $\af_{\mathbf{G4iM^w}} S$. 

The proof of the direction from right to left is straightforward because $\mathbf{G3iM^w}$ is closed under the structural rules. For weakening and contraction this is easy to see, and cut-elimination for  $\mathbf{G3iM^w}$ is proved in \cite{Olivetti}. For details, see \cite{Iemhoffb}
 
The other direction, left to right, is proved by induction on the order $\sml$ with respect to which $\mathbf{G4iM^w}$ is terminating, in a similar manner as in \cite{Iemhoffb}. So suppose $\mathbf{G3iM^w}\af S$. 
Sequents lowest in the order do not contain connectives or modal operators by definition of the weight function underlying $\sml$. Thus such sequents have to be instances of axioms, and since $\mathbf{G3iM^w}$ and $\mathbf{G4iM^w}$ have the same axioms, $S$ is provable in $\mathbf{G4iM^w}$. 

We turn to the case that $S$ is not the lowest in the order. 
If $S^a$ contains a conjunction, say $S = (\Ga,\phi_1\en\phi_2 \seq \De)$, then $S'=(\Ga,\phi_1,\phi_2\seq \De)$ is provable in $\mathbf{G3iM^w}$ by Lemma~\ref{lemimpconv}. As $\mathbf{G4iM^w}$ contains $L\en$ and $\mathbf{G4iM^w}$ is terminating, $S'\sml S$ follows. Hence $S'$ is provable in $\mathbf{G4iM^w}$ by the induction hypothesis. Thus so is $(\Ga,\phi_1\en\phi_2 \seq \De)$. A disjunction in $S^a$ as well as the case that both $p$ and $p\imp \phi$ belong to $S^a$, can be treated in the same way. 

Thus only the case that $S$ is irreducible remains, and by Lemmas~\ref{lemimpconv} and \ref{lemstrict} we may assume its proof in $\mathbf{G3iM^w}$ to be sensible and strict. The irreducibility of $S$ implies that the last inference of the proof is an application of a rule, $\rsch$, that is either a nonmodal right rule, a modal rule or $L\!\imp$. In the first two cases, $\rsch$ belongs to both calculi and the fact that $\mathbf{G4iM^w}$ is terminating implies that the premise(s) of $\rsch$ is lower in the order $\sml$ than $S$. Thus the induction hypothesis implies that the premise(s) is derivable in $\mathbf{G4iM^w}$, and since $\rsch$ belongs to $\mathbf{G4iM^w}$, the conclusion $S$ is derivable in $\mathbf{G4iM^w}$ as well. 

We turn to the third case. 
Suppose that the principal formula of the last inference is $(\gam \imp \psi)$ and $S = (\Ga,\gam\imp\psi \seq \De)$. Since the proof is sensible, $\gam$ is not atomic. 
We distinguish according to the main connective of $\gam$.  

If $\gam = \bot$, then $\Ga\seq \De$ is derivable in $\mathbf{G3iM^w}$ because of the closure under cut: $\mathbf{G3iM^w}$ derives $(\ \seq \bot \imp \psi)$, and so the cut 
\[
 \infer{\Ga\seq \De}{\ \seq \bot \imp \psi & \Ga,\bot\imp\psi \seq \De}
\]
shows that $\Ga\seq \De$ is derivable in $\mathbf{G3iM^w}$. Since $(\Ga\seq \De) \sml S$, it follows that $\Ga\seq \De$ is derivable in $\mathbf{G4iM^w}$ by the induction hypothesis. As $\mathbf{G4iM^w}$ is closed under weakening, $S$ is derivable in $\mathbf{G4iM^w}$ too. 

If $\gam = \phi_1 \en \phi_2$, then the fact that $S$ is derivable in $\mathbf{G3iM^w}$ implies the same for $S'=(\Ga,\phi_1 \imp (\phi_2\imp \psi) \seq \De)$, as $\mathbf{G3iM^w}$ is closed under cut. The fact that $\mathbf{G4iM^w}$ is terminating and contains $L\en\!\!\imp$ implies $S'\sml S$. Hence $S'$ is derivable in $\mathbf{G4iM^w}$ by the induction hypothesis. Thus so is 
$\Ga,\phi_1 \en \phi_2 \imp \psi\seq \De$ by an application of $L\en\!\!\imp$. The case that $\gam = \phi_1 \of \phi_2$ is analogous.

If $\gam = \phi_1 \imp \phi_2$, then because $\gam\imp\psi$ is the principal formula, 
both premises $S_1=(\Ga, \psi \seq \De)$ and $\Ga, \gam \imp \psi \seq \gam$ are derivable in $\mathbf{G3iM^w}$. Thus so is $\Ga, \gam \imp \psi,\phi_1 \seq \phi_2$ by the invertibility of $R\!\imp$ (Lemma~\ref{lemimpconv}). 
It is not hard to see that $\Ga,\phi_2\imp\psi,\phi_1, \phi_1 \imp \phi_2\seq \psi$ is derivable in $\mathbf{G3iM^w}$. Hence so is $\Ga, \phi_2\imp\psi,\phi_1 \seq \gam \imp \psi$. Together with $\Ga, \gam \imp \psi,\phi_1 \seq \phi_2$ and the fact that $\mathbf{G3iM^w}$ is closed under cut, this gives the derivability of $\Ga, \phi_2\imp\psi,\phi_1 \seq \phi_2$ in $\mathbf{G3iM^w}$, which implies that $S_2=\Ga, \phi_2\imp\psi \seq \gam$ is derivable in $\mathbf{G3iM^w}$. 

Since $S_1$ and $S_1$ are the premises of $L\!\imp\!\imp$ with conclusion $S$ in $\mathbf{G4iM^w}$, they both are lower in the order $\sml$ than $S$. Therefore they are derivable in $\mathbf{G4iM^w}$ by the induction hypothesis. And thus so is $S$ by an application of $L\!\imp\!\imp$. 

If $\gam = \bx\phi$, then the fact that the proof is strict and $S$ is irreducible implies that the left premise is the conclusion of the modal rule $\rsch$ with premises $\chi \seq \phi$. Thus the derivation looks as follows:
\[
 \AxiomC{$\cald_0$}\noLine
 \UnaryInfC{$\chi \seq \phi$}
 \UnaryInfC{$\Ga,\bx\phi \imp \psi,\bx\chi \seq \bx\phi$}
 \AxiomC{$\cald_2$} \noLine
 \UnaryInfC{$\Ga,\bx\chi,\psi \seq \De$}
 \BinaryInfC{$\Ga,\bx\phi\imp \psi,\bx\chi  \seq \De$}
 \DisplayProof
\]
Observe that $\mathbf{G4iM^w}$ contains the rule 
\[
 \AxiomC{$\chi \seq \phi$}
 \AxiomC{$\Ga,\bx\chi,\psi \seq \De$}
 \BinaryInfC{$\Ga,\bx\phi\imp \psi,\bx\chi  \seq \De$}
 \DisplayProof
\]
Since $\mathbf{G4iM^w}$ is terminating, it follows that both premises are below $S$ in the ordering. 
By the induction hypothesis they are derivable in $\mathbf{G4iM^w}$, say with derivations $\cald_0'$ and $\cald_2'$. Then the following is a proof of $S$ in $\mathbf{G4iM^w}$:
\[
 \AxiomC{$\cald_0'$}\noLine
 \UnaryInfC{$\chi \seq \phi$}
 \AxiomC{$\cald_2'$} \noLine
 \UnaryInfC{$\Ga,\bx\chi,\psi \seq \De$}
 \BinaryInfC{$\Ga,\bx\phi\imp \psi,\bx\chi  \seq \De$}
 \DisplayProof
\]
\end{proof}

\begin{lemma}
$\mathbf{G4iM^w}$ is closed under the structural rules, including cut. 
\end{lemma}

\begin{thm}
The systems $\mathbf{G3iM^w}$ and $\mathbf{G4iM}$ are equivalent.
\end{thm}
\begin{proof}
To show that $\mathbf{G3iM^w}$ and $\mathbf{G4iM}$ are equivalent, it suffices to show that $\mathbf{G4iM^w}$ and $\mathbf{G4iM}$ are equivalent. That every sequent derivable in $\mathbf{G4iM^w}$ is derivable in $\mathbf{G4iM}$ is clear. For the other direction it suffices to show that weakening is admissible in $\mathbf{G4iM^w}$ which already has been investigated.   
\end{proof}

\section{Uniform Lyndon interpolation}
 \label{sectionuniform}
In this section, we will prove that the logic $\mathsf{iM}$ enjoys ULIP. To this end, we will provide a stronger variant of ULIP for the sequent calculus $\mathbf{G4iM}$ and prove that the system has that property.
From now on, when we say a sequent is derivable, we mean it is derivable in $\mathbf{G4iM}$, unless specified otherwise.

\begin{theorem}\label{ThmULIPG4iM}
$\mathbf{G4iM}$ has \emph{ULIP}, i.e., for any sequent $S$, multiset $\Sigma$, atom $p$, and $\circ \in \{+, -\}$, there exist formulas $\Ap S$ and $\Ep \Sigma$ such that:
\item[$(var)$] 
$\Ap S$ and  $\Ep \Sigma$ are $p^{\circ}$-free and $V^{\dagger}(\Ap S) \subseteq V^{\dagger}(S)$ and $V^{\dagger}(\Ep \Sigma) \subseteq V^{\dagger}(\Sigma)$, for any $\dagger \in \{+, -\}$,
\begin{enumerate}
\item \label{exists1}
$\Sigma \Rightarrow \Ep \Sigma$ is derivable,
\item \label{exists2}
for any sequent $\bar{C} \Rightarrow \bar{D}$ where $\bar{D}$ has at most one formula and
$p \notin V^{\circ}(\bar{C} \Rightarrow \bar{D})$ if $\Sigma, \bar{C} \Rightarrow \bar{D}$  is derivable, then $(\Ep \Sigma, \bar{C} \Rightarrow \bar{D})$ is also derivable.
\item \label{forall1}
$S \cdot (\Ap S \Rightarrow)$ is derivable, 
\item \label{forall2}
for any multiset $\bar{C}$ such that $p \notin V^{\circ}(\bar{C})$ if $S \cdot (\bar{C} \Rightarrow)$ is derivable, then $(\bar{C}, \Epd S^a \Rightarrow \Ap S)$ is also derivable,
\end{enumerate}
$\Ap S $ (resp., $\Ep \Sigma $) is called a 
{\em uniform $\A_p^{\circ}$-interpolant of $S$} (resp., {\em uniform $\E_p^{\circ}$-interpolant of $\Sigma$}).
\end{theorem}
Let us first derive the main result of the paper as an immediate corollary.
\begin{corollary}\label{corULIP}
$\mathsf{iM}$ has ULIP. Hence, it also has UIP and LIP.
\end{corollary}
\begin{proof}
By Theorem \ref{ThmULIPG4iM}, $\mathbf{G4iM}$ has ULIP. Hence, for $\Sigma=\varnothing$, by \eqref{exists1}, we know that $(\Rightarrow \Ep \varnothing)$ is derivable.
Therefore, $\Ep \varnothing$ is provably equivalent to $\top$. Now, set $\Ap A=\Ap (\Rightarrow A)$ and $\Ep A=\Ep \{A\}$. First, by \eqref{forall1}, we know that $\Ap (\Rightarrow A) \Rightarrow A$ is derivable. Hence, $\mathsf{iM} \vdash \Ap A \to A$. Secondly, for any $p^{\circ}$-free formula $B$, if $\mathsf{iM} \vdash B \to A$, then $B \Rightarrow A$ is derivable in $\mathbf{G4iM}$. Therefore, by \eqref{forall2}, when $\bar{C}=\{B\}$, we get the derivability of $B, \Epd \varnothing \Rightarrow Ap (\Rightarrow A)$. Since $\Epd \varnothing$ is provably equivalent to $\top$, we get $B \Rightarrow \Ap (\Rightarrow A)$ and hence, $\mathsf{iM} \vdash B \to \Ap A$. Therefore, $\Ap A$ satisfies the conditions in Definition \ref{DfnUniformInterpolation}. The case for $\Ep A$ is easier and will be skipped here.
The second part of the corollary is a result of Theorem \ref{ThmULIPresultsUIP}.
\end{proof}

Now, to prove Theorem \ref{ThmULIPG4iM}, we need the following lemma.

\begin{lem}\label{LemAxiom}
$\mathbf{G4iM}$ enjoys \emph{ULIP with respect to the axioms}, i.e., for any sequent $S$, multiset $\Sigma$, atom $p$, and $\circ \in \{+, -\}$, there exist formulas $\Apax S$ and $\Epax \Sigma$ such that they satisfy conditions $(var)$, \eqref{exists1}, and \eqref{forall1} in Theorem \ref{ThmULIPG4iM} and
\begin{description}
\item[$(ii')$]
for any sequent $\bar{C} \Rightarrow \bar{D}$ such that $p \notin V^{\circ}(\bar{C} \Rightarrow \bar{D})$, if $\Sigma, \bar{C} \Rightarrow \bar{D}$ is \emph{an axiom} in $\mathbf{G4iM}$ then $(\Epax \Sigma, \bar{C} \Rightarrow \bar{D})$ is derivable,
\item[$(iv')$]
for any multiset $\bar{C}$ such that $p \notin V^{\circ}(\bar{C})$, if $S \cdot (\bar{C} \Rightarrow)$ is \emph{an axiom} in $\mathbf{G4iM}$ then $(\bar{C} \Rightarrow \Apax S)$ is derivable.
\end{description}
\end{lem}
\begin{proof}
Define $\Epax \Sigma$ as the conjunction of all $p^{\circ}$-free formulas in $\Sigma$ and $\Apax S$ as the following: if $S$ is provable, define it as $\top$, otherwise, define $\Apax S$ as the disjunction of all $p^{\circ}$-free formulas in $S^s$.  We will show that $\Epax \Sigma$ and $\Apax S$ satisfy the conditions $(var)$, \eqref{exists1} and \eqref{forall1} of Theorem \ref{ThmULIPG4iM}, and condition $(ii')$ and $(iv')$. Clearly, $\Epax \Sigma$ and $\Apax S$ are $p^{\circ}$-free, $V^{\dagger}(\Epax \Sigma) \subseteq V^{\dagger}(\Sigma)$ and $V^{\dagger}(\Apax S) \subseteq V^{\dagger}(S)$, for any $\dagger \in \{+, -\}$ and  $\Sigma \Rightarrow \Epax \Sigma$ and $S \cdot (\Apax S \Rightarrow \,) $ are derivable.

For $(ii')$, if $\Sigma, \bar{C} \Rightarrow \bar{D}$ is an axiom, it is either of the form $\Gamma, q \Rightarrow q$, where $q$ is an atom, or $\Gamma, \bot \Rightarrow \Delta$. In the first case, as $\bar{D}=\{q\}$, the atom $q$ is $p^{\circ}$-free. If $q \in \bar{C}$, then $\bar{C}, \Epax \Sigma \Rightarrow \bar{D}$ is an instance of $(Ax)$ and hence provable. If $q \in \Sigma$, then $q$ appears as a conjunct in  $\Epax \Sigma$ and $\bar{C}, \Epax \Sigma \Rightarrow \bar{D}$ is provable. If $\Sigma, \bar{C} \Rightarrow \bar{D}$ is of the form $\Gamma, \bot \Rightarrow \Delta$, then if $\bot \in \bar{C}$, then $\bar{C}, \Epax \Sigma \Rightarrow \bar{D}$ is an instance of $(L \bot)$ and provable. If $\bot \in \Sigma$, then $\bot$ appears as a conjunct in  $\Epax \Sigma$ and hence $\bar{C}, \Epax \Sigma \Rightarrow \bar{D}$ is provable.

For $(iv')$, suppose a $p^{\circ}$-free multiset $\bar{C}$ is given such that $S \cdot (\bar{C} \Rightarrow)$ is an axiom in $\mathbf{G4iM}$. If $S$ is provable, then as $\Apax S=\top$, we have $\bar{C} \Rightarrow \Apax S$. If $S$ is not provable, then there are two cases to consider. First, suppose $S \cdot (\bar{C} \Rightarrow)$ is of the form $\Gamma, q \Rightarrow q$, where $q$ is an atom. If $q \notin \bar{C}$, then $q \in S^a$ which implies that $S$ is provable. Therefore, $q \in \bar{C}$. Hence $q$ is $p^{\circ}$-free and appears as a disjunct in $\Apax S$. 
Therefore, as $q \in \bar{C}$, we get $\bar{C} \Rightarrow \Apax S$. 
Second, suppose $S \cdot (\bar{C} \Rightarrow)$ is of the form $\Gamma, \bot \Rightarrow \Delta$. If $\bot \in S^a$, then $S$ is provable. Therefore, $\bot \notin S^{a}$. Hence, $\bot \in \bar{C}$ which implies $\bar{C} \Rightarrow \Apax S$.
\end{proof}

\begin{proof}[of Theorem \ref{ThmULIPG4iM}]
Let us fix some notations. Let $S_{\psi}$ denote $(\Rightarrow \psi)$. We use $\Sigma_{q, \psi}$(respectively $\Sigma_{\Box \psi, \theta}$) to denote the multiset obtained from $\Sigma$ by replacing one instance of $q \to \psi$ (respectively $\Box \psi \to \theta$) in $\Sigma$ by $\psi$ (respectively $\theta$). Accordingly, for $S=(\Sigma \Rightarrow \Delta)$, define $S_{q, \psi}=(\Sigma_{q, \psi} \Rightarrow \Delta)$ and $S_{\Box \psi, \theta}=(\Sigma_{\Box \psi, \theta} \Rightarrow \Delta)$. 
Define $I^{\circ}_{at}(\Sigma)=\{(q, \psi) \mid q \to \psi \in \Sigma \;\text{and}\; q \;\text{is an atom and}\; p^{\circ}\text{-free}\}$ and $I_{m}(\Sigma)=\{(\phi, \psi, \theta) \mid \Box \phi, \Box \psi \to \theta \in \Sigma \;\text{and}\; \mathbf{G4iM} \vdash \phi \Rightarrow \psi\}$. Here are 
some remarks: 
\begin{description}
    \item[$(R_1)$] If $q \to \psi \in \Sigma$ (resp. $\Box \psi \to \theta \in \Sigma$), then $\Sigma_{q, \psi} \prec \Sigma$ (resp. $\Sigma_{\Box \psi, \theta} \prec \Sigma)$ and $V^{\dagger}(\Sigma_{q, \psi}) \subseteq V^{\dagger}(\Sigma)$, (resp. $V^{\dagger}(\Sigma_{\Box \psi, \theta}) \subseteq V^{\dagger}(\Sigma)$), for any $\dagger \in \{+, -\}$. Similarly, if $q \to \psi \in S^{a}$ (resp. $\Box \psi \to \theta \in S^{a}$), then $S_{q, \psi} \prec S$ (resp. $S_{\Box \psi, \theta} \prec S)$ and $V^{\dagger}(S_{q, \psi}) \subseteq V^{\dagger}(S)$, (resp. $V^{\dagger}(S_{\Box \psi, \theta}) \subseteq V^{\dagger}(S)$).
    \item[$(R_2)$] If $\Box \psi \to \theta \in \Sigma$, then $S_{\psi} \prec \Sigma$ and $V^{\dagger}(S_{\psi}) \subseteq V^{\dagger}(\Sigma)$, for any $\dagger \in \{+, -\}$. Similarly, if $\Box \psi \to \theta \in S^a$, then $S_{\psi} \prec S$ and $V^{\dagger}(S_{\psi}) \subseteq V^{\dagger}(S)$.
    \item[$(R_3)$] If $(q, \psi) \in I_{at}^{\circ}(\Sigma)$, then $q$ is $p^{\circ}$-free.
\end{description}

\noindent Define the four formulas $\exists^+ p \Sigma$, $\exists^- p \Sigma$, $\forall^+ p S$ and $\forall^- p S$ simultaneously by recursion on the well-founded order $\prec$ over the set of all multisets and sequents: if $\Sigma = \varnothing$, define $\Ep \Sigma$ as $\top$. Otherwise,
define it as: 
\begin{center}
    $\bigwedge \limits_{R \in
L \mathcal{R}} (\bigwedge \limits_i (\Ep S_i^a \to \Apd S_{i}) \to \bigvee \limits_j \Ep \Sigma_{j}) 
 \wedge (\Epax \Sigma) \wedge (\Epat \Sigma) \wedge (\Epm \Sigma)$
\end{center}
The first conjunction is over all the left rules $R$ in $\mathbf{G4w}^-$ and the rule $(Lp\!\to)$, backward applicable to $(\Sigma \Rightarrow)$, where $(\Sigma \Rightarrow)$ is the conclusion, the $S_{i}$'s are the non-contextual premises and $(\Sigma_j \Rightarrow)$'s are the contextual premises. The second conjunct is provided by Lemma \ref{LemAxiom}. The rest are defined as
\begin{center}
   $\Epat \Sigma= \bigwedge \limits_{(q, \psi) \in I^{\diamond}_{at}(\Sigma)} q \to \Ep \Sigma_{q, \psi}\quad$, 
\end{center}
\[\Epm \Sigma = \bigwedge_{\Box \phi \in \Sigma} \Box \Ep \phi \wedge \bigwedge_{\Box \psi \to \theta \in \Sigma} (\Box \Apd S_{\psi} \to \Ep \Sigma_{\Box \psi, \theta}) \wedge \bigwedge_{(\phi, \psi, \theta) \in I_{m}(\Sigma)} \Ep \Sigma_{\Box \psi, \theta}
\]
For $\Ap S$, if $S$ is provable define it as $\top$, otherwise, define $\Ap S$ as: 
\begin{center}
   $\bigvee \limits_{R} (\bigwedge \limits_i (\Epd S_i^a \to \Ap S_{i})) \vee (\Apax S) \vee  (\Apat S)
 \vee (\Apm S)$ 
\end{center}
The first disjunction is over all rules $R$ in $\mathbf{G4w}$ backward applicable to $S$, where 
$S_{i}$'s are the premises. The second disjunct is provided by Lemma \ref{LemAxiom}. The third is defined as
\begin{center}
    $\Apat S= \bigwedge\limits_{(q, \psi) \in I_{at}^{\circ}(S^a)} q \wedge (\Epd S^{a}_{q, \psi} \to \Ap S_{q, \psi})$,
\end{center}
For $\Apm S$, if $S=(\Rightarrow \Box \psi)$, define $\Apm S= \Box \Ap S_{\psi}$. Otherwise, define
\begin{center}
    $\Apm S = \bigvee\limits_{\Box \psi \to \theta \in S^{a}} (\Ap S_{\Box \psi, \theta} \wedge \Box \Ap S_{\psi}) \vee \bigvee\limits_{(\phi, \psi, \theta) \in I_m(S^{a})} \Ap S_{\Box \psi, \theta}$
\end{center}
We use induction on the well-founded order $\prec$ to prove that $\Ep \Sigma$ and 
$\Ap S$ have all the properties of Theorem \ref{ThmULIPG4iM}. 
In fact, in the induction step, we assume that for a sequent $S$ (resp. multiset $\Sigma$), all $\exists^+ p \Sigma'$, $\exists^- p \Sigma'$, $\forall^+ p S'$ and $\forall^- p S'$ exist for any sequent $S'$ and multiset $\Sigma'$ lower than $S$ (resp. $\Sigma$).\\ 
To prove that the recursive definition is well-defined, note that in any rule in $\mathbf{G4w}$, if $S_i$'s are the premises and $S$ is the conclusion, we have $S_i \prec S$. Therefore, using the remarks $(R_1)$ and $(R_2)$ above, we conclude that both $\Ap S$ and $\Ep S$ are well-defined. 

To prove $(var)$, using the induction hypothesis and the remark $(R_3)$, it is clear that $\Ap S$ and $\Ep \Sigma$ are $p^{\circ}$-free. Moreover, as every rule in $\mathbf{G4iM}$ enjoys the variable preserving property, it is enough to use the induction hypothesis, Lemma \ref{LemAxiom} and the remark $(R_1)$ and $(R_2)$ to prove $V^{\dagger}(\Ap S) \subseteq V^{\dagger}(S)$ and $V^{\dagger}(\Ep \Sigma) \subseteq V^{\dagger}(\Sigma)$, for any $\dagger \in \{+, -\}$.

To prove conditions \eqref{exists1}, \eqref{exists2}, \eqref{forall1} and \eqref{forall2}, as the cases that $\Sigma=\varnothing$ and $S$ is provable are easy, from now on, we assume that $\Sigma \neq \varnothing$ and $S$ is not provable. 

To prove \eqref{exists1}, it is enough to show that the following formulas
\[
\bigwedge \limits_{R \in
L \mathcal{R}} (\bigwedge \limits_i (\Ep S_i^a \to \Apd S_{i}) \to \bigvee \limits_j \Ep \Sigma_{j}) \; (1)\; , \;
\Epax \Sigma \; (2)\; ,  \;
\Epat \Sigma  \; (3)  \; , \; \Epm \Sigma  \; (4)
\]
are all derivable from $\Sigma$. For (1), assume that the left rule $R$ of $\mathbf{G4w^-}$ is backward applicable to $(\Sigma \Rightarrow)$ with $S_i$'s as the non-contextual premises and $(\Sigma_j \Rightarrow)$'s as the contextual premises.
Since $S_i$'s and $\Sigma_j$'s are lower than $\Sigma$, by the induction hypothesis, we have $S_i \cdot (\Apd S_{i} \Rightarrow)$, $(S_i^a \Rightarrow \Ep S^a_i)$ and $(\Sigma_j \Rightarrow \Ep \Sigma_j)$ for all $i$ and $j$. Therefore, $S_i \cdot (\bigwedge_i (\Ep S^a_{i} \to \Apd S_{i}) \Rightarrow )$ and $\Sigma_j \Rightarrow \bigvee_j \Ep \Sigma_j$ are derivable. By the free-context property of $R$, we can add $\bigwedge_i (\Ep S^a_{i} \to \Apd S_{i})$ to the antecedents of the premises and conclusion and put $\bigvee_j \Ep \Sigma_j$ in the succedents of the contextual premises and the conclusion. The rule will become:
\begin{center}
 \begin{tabular}{c c}
 \AxiomC{$\{ \bigwedge_i (\Ep S^a_{i} \to \Apd S_{i}), S_i^{a} \Rightarrow S_i^s \}_{i \in I}$}
 \AxiomC{$\{ \Sigma_j \Rightarrow \bigvee_j \Ep \Sigma_j \}_{j \in J}$}
 \BinaryInfC{$\Sigma, \bigwedge_i (\Ep S^a_{i} \to \Apd S_{i}) \Rightarrow \bigvee_j \Ep \Sigma_j$}
 \DisplayProof
\end{tabular}
\end{center}
Hence, we get $\Sigma \Rightarrow (\bigwedge_i (\Ep S^a_{i} \to \Apd S_{i}) \rightarrow \bigvee_j \Ep \Sigma_j)$. The case where the last rule is $(Lp\!\to)$ is similar.

For (2), we use Lemma \ref{LemAxiom}. 
For (3), if $q$ is $p^{\diamond}$-free and $q \to \psi \in \Sigma$, then we have $\Sigma=\Sigma' \cup \{q \to \psi \}$ and $\Sigma_{q, \psi}=\Sigma' \cup \{\psi\}$. As $\Sigma_{q, \psi}$ is lower than $\Sigma$, by the induction hypothesis $\Sigma_{q, \psi} \Rightarrow \Ep \Sigma_{q, \psi}$. Therefore, $\Sigma', \psi \Rightarrow \Ep \Sigma_{q, \psi}$ which implies $\Sigma', q, q \to \psi \Rightarrow \Ep \Sigma_{q, \psi}$. Hence, we get $\Sigma \Rightarrow q \to \Ep \Sigma_{q, \psi}$. 

For (4), we have to show that each conjunct in $\Epm \Sigma$ is derivable from $\Sigma$. For the first conjunct, suppose $\Box \phi \in \Sigma$, i.e., $\Sigma=\Pi, \Box \phi$. As $\phi$ is lower than $\Sigma$, by the induction hypothesis we have $\phi \Rightarrow \Ep \phi$. Therefore, by the rule $(M)$, we have $\Box \phi \Rightarrow \Box \Ep \phi$. By $(L w)$, we get $\Pi, \Box \phi \Rightarrow \Box \Ep \phi$, which is $\Sigma \Rightarrow \Box \Ep \phi$. 
For the second conjunct, suppose $\Box \psi \to \theta \in \Sigma$, i.e., $\Sigma=\Pi, \Box \psi \to \theta$. As $\Sigma_{\Box \psi, \theta}$ and $S_{\psi}$ are lower than $\Sigma$, by the induction hypothesis, we have $\Sigma_{\Box \psi, \theta} \Rightarrow \Ep \Sigma_{\Box \psi, \theta}$ and $S_{\psi} \cdot (\Apd S_{\psi} \Rightarrow)$, which are $\Pi, \theta \Rightarrow \Ep \Sigma_{\Box \psi, \theta}$ and $\Apd S_{\psi} \Rightarrow \psi$. By $(L w)$, we have $\Pi, \Box \Apd S_{\psi}, \theta \Rightarrow \Ep \Sigma_{\Box \psi, \theta}$. Applying the rule $(LM\!\!\to)$, we get $\Pi, \Box \Apd S_{\psi} , \Box \psi \to \theta \Rightarrow \Ep \Sigma_{\Box \psi, \theta}$, which implies $\Sigma \Rightarrow \Box \Apd S_{\psi} \to \Ep \Sigma_{\Box \psi, \theta}$.
For the last conjunct, as $(\phi, \psi, \theta) \in I_m(\Sigma)$, we know that $\Box \phi, \Box \psi \to \theta \in \Sigma$ and the sequent $\phi \Rightarrow \psi$ is provable. Therefore, $\Sigma= \Pi, \Box \phi, \Box \psi \to \theta$. As $\Sigma_{\Box \psi, \theta}$ is lower than $\Sigma$, by the induction hypothesis, we have $\Sigma_{\Box \psi, \theta} \Rightarrow \Ep \Sigma_{\Box \psi, \theta}$, which is $\Pi, \Box \phi , \theta \Rightarrow \Ep \Sigma_{\Box \psi, \theta}$. Since $\phi \Rightarrow \psi$ is also provable, we can apply the rule $(LM \!\!\to)$ to obtain $\Pi, \Box \phi, \Box \psi \to \theta \Rightarrow \Ep \Sigma_{\Box \psi, \theta}$, which is $\Sigma \Rightarrow \Ep \Sigma_{\Box \psi, \theta}$.\\

For \eqref{exists2}, we assume $p \notin V^{\diamond}(\bar{C}), p \notin V^{\circ}(\bar{D})$ and $\Sigma, \bar{C} \Rightarrow \bar{D}$ is derivable and we use induction on the length of its proof. If $\Sigma, \bar{C} \Rightarrow \bar{D}$ is an axiom, we have $\Epax \Sigma, \bar{C} \Rightarrow \bar{D}$, by Lemma \ref{LemAxiom}, and hence $\Ep \Sigma, \bar{C} \Rightarrow \bar{D}$. If the last rule is a left rule in $\mathbf{G4w}^-$, it is of the form:
\begin{center}
 \begin{tabular}{c c}
 \AxiomC{$\{ \Gamma, \bar{\phi}_i \Rightarrow \bar{\delta}_i \}_{i \in I}$}
 \AxiomC{$\{ \Gamma, \bar{\psi}_{j} \Rightarrow \bar{D} \}_{j \in J}$}
 \BinaryInfC{$\Gamma, \phi \Rightarrow \bar{D}$}
 \DisplayProof
\end{tabular}
\end{center}
Then, there are two cases to consider, i.e., either $\phi \in \bar{C}$ or $\phi \in \Sigma$. If $\phi \in \bar{C}$, set $\bar{C}'=\bar{C}-\{\phi\}$. Since $\phi \in \bar{C}$, it is $p^{\diamond}$-free by the assumption, and by the local variable preserving property $\bar{\phi}_i$'s and $\bar{\psi}_j$'s are $p^{\diamond}$-free and $\bar{\delta}_i$'s are $p^{\circ}$-free. By the induction hypothesis, as $p \notin V^{\circ}(\bar{C}', \bar{\phi}_i \Rightarrow \bar{\delta}_i)$, $p \notin V^{\circ}(\bar{C}', \bar{\psi}_j \Rightarrow \bar{D})$, and $(\Sigma, \bar{C}', \bar{\phi}_i \Rightarrow \bar{\delta}_i)$ and $(\Sigma, \bar{C}', \bar{\psi}_j \Rightarrow \bar{D})$ have shorter proofs, we have $(\Ep \Sigma, \bar{C}', \bar{\phi}_i \Rightarrow \bar{\delta}_i)$ and $(\Ep \Sigma, \bar{C}', \bar{\psi}_j \Rightarrow \bar{D})$ are derivable, for each $i \in I$ and $j \in J$. By using the rule itself, we get $\Ep \Sigma, \bar{C}', \phi \Rightarrow \bar{D}$, which implies $\Ep \Sigma, \bar{C} \Rightarrow \bar{D}$.

If $\phi \in \Sigma$, set $\Sigma'=\Sigma-\{\phi\}$. Hence, the last rule is of the form:
\begin{center}
 \begin{tabular}{c c}
 \AxiomC{$\{ \Sigma', \bar{C}, \bar{\phi}_i \Rightarrow \bar{\delta}_i \}_{i \in I}$}
 \AxiomC{$\{ \Sigma', \bar{C}, \bar{\psi}_j \Rightarrow \bar{D} \}_{j \in J}$}
 \BinaryInfC{$\Sigma', \bar{C}, \phi \Rightarrow \bar{D}$}
 \DisplayProof
\end{tabular}
\end{center}
Note that neither $\bar{C}$ nor $\bar{D}$ contain any active formulas. By the free-context property, if we delete $\bar{C}$ and $\bar{D}$ from the premises and conclusion of the last rule, the rule remains valid and it changes to:
\begin{center}
 \begin{tabular}{c c}
 \AxiomC{$\{ \Sigma', \bar{\phi}_i \Rightarrow \bar{\delta}_i \}_{i \in I}$}
 \AxiomC{$\{ \Sigma', \bar{\psi}_j \Rightarrow  \}_{j \in J}$}
 \BinaryInfC{$\Sigma', \phi \Rightarrow $}
 \DisplayProof
\end{tabular}
\end{center}
Hence, the rule is backward applicable to 
$(\Sigma \Rightarrow)$. Set $S_i=(\Sigma', \bar{\phi}_i \Rightarrow \bar{\delta}_i)$ and $\Sigma_j=\Sigma', \bar{\psi}_j$. As $S_i \cdot (\bar{C} \Rightarrow)$ and $(\Sigma_j, \bar{C} \Rightarrow \bar{D})$ are provable and $S_i$'s and $\Sigma_j$'s are lower than $(\Sigma \Rightarrow)$, by the induction hypothesis, we have $(\Ep S^a_i, \bar{C} \Rightarrow \Apd S_i)$ and $(\Ep \Sigma_j, \bar{C} \Rightarrow \bar{D})$. Hence $\bar{C} \Rightarrow \bigwedge_i (\Ep S^a_i \to \Apd S_i)$ and $(\bigvee_j \Ep \Sigma_j, \bar{C} \Rightarrow \bar{D})$ are derivable. Therefore, we have $(\bigwedge_i (\Ep S^a_i \to \Apd S_i) \to \bigvee_j \Ep \Sigma_j, \bar{C} \Rightarrow \bar{D})$. As $\bigwedge_i (\Ep S^a_i \to \Apd S_i) \to \bigvee_j \Ep \Sigma_j$ is a conjunct in $\Ep \Sigma$, we have $\Ep \Sigma, \bar{C} \Rightarrow \bar{D}$.\\

If the last rule of the proof is a right rule, then it is of the form:
\begin{center}
 \begin{tabular}{c c}
 \AxiomC{$\{ \Sigma, \bar{C}, \bar{\phi}_i \Rightarrow \bar{\psi}_i \}_{i \in I}$}
 \UnaryInfC{$\Sigma, \bar{C} \Rightarrow \phi$}
 \DisplayProof
\end{tabular}
\end{center}
and $\bar{D}=\{\phi\}$. Hence, $\phi$ is $p^{\circ}$-free. By the local variable preserving property, $\bar{\phi}_i$'s are $p^{\diamond}$-free and $\bar{\psi}_i$'s are $p^{\circ}$-free. By the induction hypothesis, as $(\bar{C}, \bar{\phi}_i \Rightarrow \bar{\psi}_i)$ is $p^{\circ}$-free and $(\Sigma, \bar{C}, \bar{\phi}_i \Rightarrow \bar{\psi}_i)$ has a shorter proof, we have $(\Ep \Sigma, \bar{C}, \bar{\phi}_i \Rightarrow \bar{\psi}_i)$. Using the rule itself, we get $\Ep \Sigma, \bar{C} \Rightarrow \phi$
which is $\Ep \Sigma, \bar{C} \Rightarrow \bar{D}$.\\

If the last rule is the rule $(Lp\! \to)$, then it is of the form 
\begin{center}
 \begin{tabular}{c c}
 \AxiomC{$\Gamma, q, \psi \Rightarrow \bar{D}$}
 \UnaryInfC{$\Gamma, q, q \to \psi \Rightarrow \bar{D}$}
 \DisplayProof
\end{tabular}
\end{center}
There are four cases to consider, depending on whether $q$ or $q \to \psi$ are in $\bar{C}$. 

If $q, q \to \psi \in \bar{C}$, then set $\bar{C'}= \bar{C}-\{q, q\to \psi\}$. As $\Sigma, \bar{C'}, q, \psi \Rightarrow \bar{D}$ has a shorter proof and $q$ and $\psi$ are $p^{\diamond}$-free, then by the induction hypothesis, $\Ep \Sigma, \bar{C}', q, \psi \Rightarrow \bar{D}$. Hence, by the rule itself, we have $\Ep \Sigma, \bar{C}', q, q \to \psi \Rightarrow \bar{D}$.

If $q, q \to \psi \notin \bar{C}$, then 
the premise of the rule is of the form $\Sigma_{q, \psi}, \bar{C} \Rightarrow \bar{D}$, and by the induction hypothesis, we have $\Ep \Sigma_{q, \psi}, \bar{C} \Rightarrow \bar{D}$. However, by the free-context property, we can delete $\bar{C}$ and $\bar{D}$ in the premise and conclusion and the rule remains valid and has the form 
 \begin{tabular}{c c}
 \AxiomC{$\Sigma_{ q, \psi} \Rightarrow$}
 \UnaryInfC{$\Sigma \Rightarrow $}
 \DisplayProof.
\end{tabular}
Therefore, this rule is backward applicable to $(\Sigma \Rightarrow)$ and 
as this rule has no non-contextual premise, $\top \to \Ep \Sigma_{q, \psi}$ appears as a conjunct in $\Ep \Sigma$. Hence, $\Ep \Sigma, \bar{C} \Rightarrow \bar{D}$ is derivable.

If $q \to \psi \notin \bar{C}$ and $q \in \bar{C}$, then $q$ is $p^{\diamond}$-free. Set $\bar{C'}= \bar{C}-\{q\}$. Then, the premise of the rule is of the form $\Sigma_{q, \psi}, q, \bar{C}' \Rightarrow \bar{D}$. As $q \to \psi \in \Sigma$, we have $\Sigma_{q, \psi} \prec \Sigma$, by $(R_1)$. As $(\bar{C}', q \Rightarrow \bar{D} )$ is $p^{\circ}$-free, by the induction hypothesis, $\bar{C}', q, \Ep \Sigma_{q, \psi} \Rightarrow \bar{D} $. Using the rule $(Lp\!\to)$, we get $\bar{C}', q, q \to \Ep \Sigma_{q, \psi} \Rightarrow \bar{D}$. As $q \to \Ep \Sigma_{q, \psi}$ is a conjunct in $\Ep \Sigma$, we have $\Ep \Sigma, \bar{C} \Rightarrow \bar{D}$.

If $q \notin \bar{C}$ and $q \to \psi \in \bar{C}$, then $\psi$ is $p^{\diamond}$-free and $q$ is $p^{\circ}$-free. Set $\bar{C'}=\bar{C}-\{q \to \psi\}$. Then, the premise of the rule is of the form $\Sigma, \psi, \bar{C}' \Rightarrow \bar{D}$.
By the induction hypothesis, we have $\Ep \Sigma, \psi, \bar{C}' \Rightarrow \bar{D} $. Moreover, by $(Ax)$, we have $\Sigma \Rightarrow q$. Since $q$ is $p^{\circ}$-free, we have $\Epax \Sigma \Rightarrow q$ and hence, $\Ep \Sigma \Rightarrow q$. Therefore, $\Ep \Sigma, q \to \psi, \bar{C}' \Rightarrow \bar{D}$, which is $\Ep \Sigma, \bar{C} \Rightarrow \bar{D}$.\\


If the last rule in the proof is the modal rule $(M)$, then it is of the form
\begin{center}
\begin{tabular}{c}
 \AxiomC{$\phi \Rightarrow \psi$} 
   \RightLabel{\small$M$} 
 \UnaryInfC{$\Box \phi \Rightarrow \Box \psi$}
 \DisplayProof
\end{tabular}
\end{center}
As $\Sigma \neq \varnothing$, we have $\bar{C}= \varnothing$ and $\Sigma=\Box \phi$. 
Since $\bar{D}= \Box \psi$, the formula $\psi$ is $p^{\circ}$-free. As $(\phi \Rightarrow \psi)$ is provable and $\phi$ is lower than $\Sigma$, by the induction hypothesis, we have $\Ep \phi \Rightarrow \psi$ and by $(M)$, $\Box \Ep \phi \Rightarrow \Box \psi$. As $\Box \Ep \phi$ appears as a conjunct in the definition of $\Ep \Sigma$, we get $\Ep \Sigma, \bar{C} \Rightarrow \bar{D}$.

If the last rule in the proof is the modal rule $(LM\!\!\to)$, then it is of the form

\begin{center}
\begin{tabular}{c}
 \AxiomC{$\phi \Rightarrow \psi$} 
 \AxiomC{$\Gamma, \Box \phi, \theta \Rightarrow \bar{D}$} 
   \RightLabel{\small$LM \to$} 
 \BinaryInfC{$\Gamma, \Box \phi, \Box \psi \to \theta \Rightarrow \bar{D}$}
 \DisplayProof
\end{tabular}
\end{center}
There are four cases to consider based on which formulas are in $\bar{C}$. 

If $\Box \phi, \Box \psi \to \theta \in \bar{C}$, then $\phi$ and $\theta$ are $p^{\diamond}$-free and $\psi$ is $p^{\circ}$-free. Set $\bar{C'}= \bar{C}-\{\Box \phi, \Box \psi \to \theta\}$. The right premise of the rule is of the form $\Sigma, \bar{C'}, \Box \phi, \theta \Rightarrow \bar{D}$. Hence, by the induction hypothesis $\Ep \Sigma, \bar{C'}, \Box \phi, \theta \Rightarrow \bar{D}$. By $(LM\!\! \to)$ on the latter sequent and $\phi \Rightarrow \psi$, we get $\Ep \Sigma, \bar{C'}, \Box \phi, \Box \psi \to \theta \Rightarrow \bar{D}$ which is $\Ep \Sigma, \bar{C} \Rightarrow \bar{D}$.

If $\Box \phi, \Box \psi \to \theta \notin \bar{C}$, then set $\Sigma'=\Sigma-\{\Box \phi, \Box \psi \to \theta\}$. Hence,
the right premise of the rule is of the form $\Sigma', \Box \phi, \theta, \bar{C} \Rightarrow \bar{D}$, or equivalently $\Sigma_{\Box \psi, \theta}, \bar{C} \Rightarrow \bar{D}$. Therefore, by the induction hypothesis $\Ep \Sigma_{\Box \psi, \theta}, \bar{C} \Rightarrow \bar{D}$. However, $\Ep \Sigma_{\Box \psi, \theta}$ appears as a conjunct in the definition of $\Ep \Sigma$, since $\Box \phi, \Box \psi \to \theta \in \Sigma$ and $\mathbf{G4iM} \vdash \phi \Rightarrow \psi$. Consequently, $\Ep \Sigma , \bar{C} \Rightarrow \bar{D}$.

If $\Box \psi \to \theta \in \bar{C}$ and $\Box \phi \notin \bar{C}$, then $\theta$ is $p^{\diamond}$-free and $\psi$ is $p^{\circ}$-free. Set $\bar{C}'=\bar{C}-\{\Box \psi \to \theta\}$. Then, the premise has the form $\Sigma, \bar{C}', \theta \Rightarrow \bar{D}$. By the induction hypothesis, we have $\Ep \Sigma, \bar{C}', \theta \Rightarrow \bar{D}$ and by $(Lw)$, we have $\Ep \Sigma, \bar{C}', \Box \Ep \phi, \theta \Rightarrow \bar{D}$.
Moreover, as $\phi \Rightarrow \psi$ is provable, $\psi$ is $p^{\circ}$-free, and $\phi$ is lower than $\Sigma$, by the induction hypothesis, we have $\Ep \phi \Rightarrow \psi$. Applying $(LM \!\!\to)$ on $\Ep \phi \Rightarrow \psi$ and  $\Ep \Sigma, \bar{C}', \Box \Ep \phi, \theta \Rightarrow \bar{D}$, we get $\Ep \Sigma, \bar{C}', \Box \Ep \phi, \Box \psi \to \theta \Rightarrow \bar{D}$. Note that as $\Box \phi \in \Sigma$, by definition $\Box \Ep \phi$ appears as a conjunct in $\Ep \Sigma$. Therefore, $\Ep \Sigma, \bar{C} \Rightarrow \bar{D}$.

Finally, if $\Box \phi \in \bar{C}$ and $\Box \psi \to \theta \notin \bar{C}$, then 
$\phi$ is $p^{\diamond}$-free. Set $\bar{C}'=\bar{C}-\{\Box \phi\}$. Therefore, the right premise is of the form $\bar{C}', \Box \phi, \Sigma_{\Box \psi, \theta} \Rightarrow \bar{D}$. Since $\bar{C}'$ and $\phi$ are $p^{\diamond}$-free and $\bar{D}$ is $p^{\circ}$-free, by the induction hypothesis, we have $\bar{C}', \Box \phi, \Ep \Sigma_{\Box \psi, \theta} \Rightarrow \bar{D}$. Moreover, for the premise $\phi \Rightarrow \psi$, as $S_{\psi}$ is lower than $\Sigma$ and $\phi$ is $p^{\diamond}$-free, by the induction hypothesis, we get $\phi, \Ep S_{\psi}^{a} \Rightarrow \Apd S_{\psi}$. However, as $S_{\psi}^{a}= \varnothing$, by definition we have $\Ep S_{\psi}^{a} = \top$. Hence, $\phi \Rightarrow \Apd S_{\psi}$. Applying $(LM \!\!\to)$ on $\phi \Rightarrow \Apd S_{\psi}$ and $\bar{C}', \Box \phi, \Ep \Sigma_{\Box \psi, \theta} \Rightarrow \bar{D}$ we get $\bar{C}', \Box \phi, \Box \Apd S_{\psi} \to \Ep \Sigma_{\Box \psi, \theta} \Rightarrow \bar{D}$. Therefore, as $\Box \psi \to \theta \in \Sigma$, the formula $\Box \Apd S_{\psi} \to \Ep \Sigma_{\Box \psi, \theta}$ is a conjunct in $\Ep \Sigma$, and hence $\Ep \Sigma , \bar{C} \Rightarrow \bar{D}$ is derivable.\\

To prove \eqref{forall1}, it is enough to show that the following are provable:
\begin{center}
    \begin{tabular}{cc}
        $S \cdot (\bigwedge \limits_i (\Epd S_i^a \to \Ap S_{i}) \Rightarrow) \quad (1),$ &  $ S \cdot (\Apax S \Rightarrow)  \quad (2),$\\
        $S \cdot (\Apat S \Rightarrow) \quad (3),$ & $S \cdot (\Apm S \Rightarrow) \quad (4).$
    \end{tabular}
\end{center}
For (1), assume that the rule $R$ in $\GthW$ is backward applicable to $S$ and the premises of $R$ are $S_i$'s.
As $S_i$'s are lower than $S$, by the induction hypothesis $S_i \cdot (\Ap S_{i} \Rightarrow)$ and $S_i^a \Rightarrow \Epd S_i^a$. Therefore, $S_i \cdot (\Epd S_i^a \to \Ap S_{i} \Rightarrow)$. Hence, by weakening, we have $S_i \cdot (\{\Epd S_i^a \to \Ap S_{i}\}_i \Rightarrow )$. Since any rule in $\mathbf{G4w}$ has the free-context property, we can add $\{\Epd S_i^a \to \Ap S_{i}\}_i$ to the antecedents of the premises and conclusion and by the rule itself, we have $S \cdot (\{\Epd S_i^a \to \Ap S_{i}\}_i \Rightarrow )$ and hence we get $S \cdot (\bigwedge_i (\Epd S_i^a \to \Ap S_{i}) \Rightarrow \,)$.

For (2), see Lemma \ref{LemAxiom}. For (3), if $(q, \psi) \in I^{\circ}_{at}(S^{a})$, then $S= (\Gamma, q \to \psi \Rightarrow \Delta)$ and $S_{q, \psi}= (\Gamma , \psi \Rightarrow \Delta)$. As $S_{q, \psi} \prec S$, by the induction hypothesis $\Gamma, \psi, \Ap S_{q, \psi} \Rightarrow \Delta$ and $\Gamma, \psi \Rightarrow \Epd S^{a}_{q, \psi}$. Hence, $\Gamma, \psi, \Epd S_{q, \psi}^a \to \Ap S_{q, \psi} \Rightarrow \Delta$. Therefore, $\Gamma, q, q \to \psi, \Epd S_{q, \psi}^a \to \Ap S_{q, \psi} \Rightarrow \Delta$ which implies $\Gamma, q \to \psi, q \wedge (\Epd S_{q, \psi}^a \to \Ap S_{q, \psi}) \Rightarrow \Delta$, and we get $S \cdot (q \wedge (\Epd S_{q, \psi}^a \to \Ap S_{q, \psi}) \Rightarrow \,)$.

For (4), if $S=(\Rightarrow \Box \psi)$, then by definition $\Apm S= \Box \Ap S_{\psi}$. As $S_{\psi} \prec S$, by the induction hypothesis $S_{\psi} \cdot (\Ap S_{\psi} \Rightarrow \,)= (\Ap S_{\psi} \Rightarrow \psi)$ is provable. By $(M)$ we get $\Box \Ap S_{\psi} \Rightarrow \Box \psi$ which is $S \cdot (\Apm S \Rightarrow \,)$. If $S$ is not of the form $(\Rightarrow \Box \psi)$, then $\Apm S$ is defined by a disjunction over two families of formulas. We have to show that adding any such disjunct to the antecedent of $S$ makes it provable.

For the first family of disjuncts, if $\Box \psi \to \theta \in S^a$, then $S$ is in form $(\Gamma, \Box \psi \to \theta \Rightarrow \Delta)$. As $S_{\psi}$ and $S_{\Box \psi, \theta}$ are lower than $S$, by the induction hypothesis we have $(\Ap S_{\psi} \Rightarrow \psi)$ and $(\Ap S_{\Box \psi, \theta}, \Gamma , \theta \Rightarrow \Delta)$. By $(Lw)$, we get $(\Box \Ap S_{\psi}, \Ap S_{\Box \psi, \theta}, \Gamma , \theta \Rightarrow \Delta)$ and by $(LM \to)$, $\Box \Ap S_{\psi}, \Ap S_{\Box \psi, \theta}, \Gamma, \Box \psi \to \theta \Rightarrow \Delta$, which implies $S \cdot (\Box \Ap S_{\psi} \wedge \Ap S_{\Box \psi, \theta} \Rightarrow \,)$.

For the second family of disjuncts, if $(\phi, \psi, \theta) \in I_m(S^a)$, then $S$ has the form $S= (\Gamma, \Box \phi, \Box \psi \to \theta \Rightarrow \Delta)$ and $\mathbf{G4iM \vdash \phi \Rightarrow \psi}$. Since $S_{\Box \psi, \theta}=(\Gamma, \Box \phi, \theta \Rightarrow \Delta)$ is lower than $S$, by the induction hypothesis we have $\Ap S_{\Box \psi, \theta}, \Gamma, \Box \phi, \theta \Rightarrow \Delta$. Applying $(LM \to)$ on the latter sequent and $\phi \Rightarrow \psi$, we get $S \cdot (\Ap S_{\Box \psi, \theta} \Rightarrow)$.\\

For \eqref{forall2}, we use induction on the length of the proof of $S \cdot (\bar{C} \Rightarrow)$. If $S \cdot (\bar{C} \Rightarrow)$ is an axiom, by Lemma \ref{LemAxiom} we have $\bar{C} \Rightarrow \Apax S$, and hence $\Epd S^a, \bar{C} \Rightarrow \Ap S$. If the last rule is a left one in $\mathbf{G4w}^-$, it is of the form:
\begin{center}
 \begin{tabular}{c c}
 \AxiomC{$\{ \Gamma, \bar{\phi}_{i} \Rightarrow \bar{\delta}_i \}_{i \in I}$}
 \AxiomC{$\{ \Gamma, \bar{\psi}_j \Rightarrow \Delta \}_{j \in J}$}
 \BinaryInfC{$\Gamma, \phi \Rightarrow \Delta$}
 \DisplayProof
\end{tabular}
\end{center}
There are two cases to consider, either $\phi \in \bar{C}$ or $\phi \in S^a$. If $\phi \in \bar{C}$, then it is $p^{\circ}$-free and by the local variable preserving property, $\bar{\phi}_i$'s and $\bar{\psi}_j$'s are $p^{\circ}$-free and $\bar{\delta}_i$'s are $p^{\diamond}$-free. Set $\bar{C}'=\bar{C}-\{\phi\}$. As the sequent $S \cdot (\bar{C}', \bar{\psi}_j \Rightarrow)$ has a shorter proof, by the induction hypothesis we have $(\Epd S^a, \bar{C}', \bar{\psi}_j \Rightarrow \Ap S)$. Now, we want to use the induction hypothesis to prove $(\Epd S^a, \bar{C}', \bar{\phi}_i \Rightarrow \bar{\delta}_i)$. Note that it might be the case that $S^s=\varnothing$ and hence $S^{a}$ is not lower than $S$. However, we already saw that for any multiset $\Sigma$, having the conditions \eqref{exists1}, \eqref{exists2}, \eqref{forall1} and \eqref{forall2} for all multisets and sequents below $\Sigma$ proves part \eqref{exists2} for $\Sigma$. Putting $\Sigma=S^a$, as any multiset or sequent below $S^a$ is also below $S$, by the induction hypothesis we have all four conditions for them and hence we have \eqref{exists2} for $S^a$. Now, as $(S^a, \bar{C}', \bar{\phi}_i \Rightarrow \bar{\delta}_i)$ is provable and $p \notin V^{\diamond}(\bar{C}', \bar{\phi}_i \Rightarrow \bar{\delta}_i)$, by \eqref{exists2}, we have $(\Epd S^a, \bar{C}', \bar{\phi}_i \Rightarrow \bar{\delta}_i)$.
By the rule itself, we have  
\begin{center}
 \begin{tabular}{c c}
 \AxiomC{$ \{\Epd S^a, \bar{C}', \bar{\phi}_i \Rightarrow \bar{\delta}_i\}_{i \in I}$}
 \AxiomC{$\{\Epd S^a, \bar{C}', \bar{\psi}_j \Rightarrow \Ap S \}_{j \in J}$}
 \BinaryInfC{$\Epd S^a, \bar{C}', \phi \Rightarrow \Ap S$}
 \DisplayProof
\end{tabular}
\end{center}
which is $\Epd S^a, \bar{C} \Rightarrow \Ap S$.\\
If $\phi \notin \bar{C}$, then it does not contain any active formulas of the rule. Set $\Gamma'=S^a-\{\phi\}$. The last rule is of the form:
\begin{center}
 \begin{tabular}{c c}
 \AxiomC{$\{ \Gamma',  \bar{C}, \bar{\phi}_i \Rightarrow \bar{\delta}_i \}_{i \in I}$}
 \AxiomC{$\{ \Gamma',  \bar{C}, \bar{\psi}_j \Rightarrow \Delta \}_{j \in J}$}
 \BinaryInfC{$\Gamma', \bar{C}, \phi \Rightarrow \Delta$}
 \DisplayProof
\end{tabular}
\end{center}
By the free-context property, we can delete $\bar{C}$ from the premises and conclusion of the rule which remains valid and changes to:
\begin{center}
 \begin{tabular}{c c}
 \AxiomC{$\{ \Gamma', \bar{\phi}_i \Rightarrow \bar{\delta}_i \}_{i \in I}$}
 \AxiomC{$\{ \Gamma', \bar{\psi}_j \Rightarrow \Delta \}_{j \in J}$}
 \BinaryInfC{$\Gamma', \phi \Rightarrow \Delta$}
 \DisplayProof
\end{tabular}
\end{center}
Therefore, the rule is backward applicable to $S=(\Gamma', \phi \Rightarrow \Delta)$. Set $S_i=(\Gamma', \bar{\phi}_i \Rightarrow \bar{\delta}_i)$ and $T_j=(\Gamma' , \bar{\psi}_j \Rightarrow \Delta)$. As $S_i$'s and $T_j$'s are lower than $S$ and $S_i \cdot (\bar{C} \Rightarrow)$ and $T_j \cdot (\bar{C} \Rightarrow)$ are provable, by the induction hypothesis, we have $\Epd S^a_i, \bar{C} \Rightarrow \Ap S_i$ and $\Epd T^a_j, \bar{C} \Rightarrow \Ap T_j$. Hence, $\bar{C} \Rightarrow \bigwedge_i (\Epd S^a_i \to \Ap S_i) \wedge \bigwedge_j (\Epd T^a_j \to \Ap T_j)$ and as $\bigwedge_i (\Epd S^a_i \to \Ap S_i) \wedge \bigwedge_j (\Epd T^a_j \to \Ap T_j)$ appears as a disjunct in $\Ap S$, we have $\bar{C} \Rightarrow \Ap S$ and hence $\Epd S^a, \bar{C} \Rightarrow \Ap S$.

\noindent The case where the last rule is a right one in $\mathbf{G4w^-}$ is similar.
 If the last rule is the rule $(Lp\!\to)$, then it has the form 
\begin{center}
 \begin{tabular}{c c}
 \AxiomC{$\Gamma, q, \psi \Rightarrow \Delta$}
 \UnaryInfC{$\Gamma, q, q \to \psi \Rightarrow \Delta$}
 \DisplayProof
\end{tabular}
\end{center}
There are four cases to consider, depending on whether $q$ or $q \to \psi$ are in $\bar{C}$. 

If $q, q \to \psi \in \bar{C}$, then $q$ and $\psi$ are $p^{\circ}$-free. Set $\bar{C'}= \bar{C}-\{q, q\to \psi\}$. As the premise $S \cdot (\bar{C'}, q, \psi \Rightarrow)$ has a shorter proof, by the induction hypothesis $\Epd S^a, \bar{C}', q, \psi \Rightarrow \Ap S$ and by $(Lp\!\to)$ we have $\Epd S^a, \bar{C}', q, q \to \psi \Rightarrow \Ap S$.

If $q, q \to \psi \notin \bar{C}$, then by the free-context property we can delete $\bar{C}$ from the premise and the conclusion and the rule remains valid and changes to
\begin{center}
 \begin{tabular}{c c}
 \AxiomC{$\Gamma-\bar{C}, q, \psi \Rightarrow \Delta$}
 \UnaryInfC{$\Gamma-\bar{C}, q, q \to \psi \Rightarrow \Delta$}
 \DisplayProof
\end{tabular}
\end{center}
Note that the conclusion is $S$. Therefore, the rule is backward applicable to $S$.  Denote the premise by $S'$. By the induction hypothesis we have $\Epd S'^a, \bar{C} \Rightarrow \Ap S'$. As $\Epd S'^a \to \Ap S'$ appears as a disjunct in $\Ap S$, we have $\Epd S^a, \bar{C} \Rightarrow \Ap S$.

If $q \to \psi \notin \bar{C}$ and $q \in \bar{C}$, 
then $q \to \psi \in S^a$, $q$ is $p^{\circ}$-free and the premise is of the form $S_{q, \psi}\cdot(\bar{C}\Rightarrow)$. By $(R_1)$, we know $S_{q, \psi} \prec S$. Hence, by the induction hypothesis $\Epd S_{q, \psi}^a, \bar{C} \Rightarrow \Ap S_{q, \psi}$. Hence, as $q \in \bar{C}$, we get $\bar{C} \Rightarrow q \wedge (\Epd S_{q, \psi}^a \to \Ap S_{q, \psi})$. As $q \to \psi \in S^a$ and $q$ is $p^{\circ}$-free, we have $(q, \psi) \in I^{\circ}_{at} (S^{a})$. Hence, $q \wedge (\Epd S_{q, \psi}^a \to \Ap S_{q, \psi})$ is a disjunct in $\Ap S$, and we have $\Epd S^a, \bar{C} \Rightarrow \Ap S$.

If $q \notin \bar{C}$ and $q \to \psi \in \bar{C}$, then $q \in S^a$, $q$ is $p^{\diamond}$-free and $\psi$ is $p^{\circ}$-free. Set $\bar{C'}=\bar{C}-\{q \to \psi\}$.
As the premise is of the form $S \cdot (\bar{C'}, \psi \Rightarrow \,)$ and it has a shorter proof, by the induction hypothesis we have $\Epd S^a, \bar{C}', \psi \Rightarrow \Ap S$. As $S^a \Rightarrow q$ is an instance of $(Ax)$ and $q$ is $p^{\diamond}$-free, we reach $\Epdax S^a \Rightarrow q$. Therefore, as $\Epdax S^a$ is a conjunct in $\Epd S^a$, we have $\Epd S^a, q \to \psi, \bar{C}' \Rightarrow  \Ap S$.

If the last rule in the proof is the modal rule $(M)$, then it is of the form
\begin{center}
\begin{tabular}{c}
 \AxiomC{$\phi \Rightarrow \psi$} 
   \RightLabel{\small$M$} 
 \UnaryInfC{$\Box \phi \Rightarrow \Box \psi$}
 \DisplayProof
\end{tabular}
\end{center}
If $\bar{C} =\varnothing$, then $S=(\Box \phi \Rightarrow \Box \psi)$ is provable which contradicts with the assumption that $S$ is not provable. Hence, 
$\bar{C}=\Box \phi$. Therefore, $\phi$ is $p^{\circ}$-free and $S=(\Rightarrow \Box \psi)$. By definition $\Apm S= \Box \Ap S_{\psi}$. Since $\phi$ is $p^{\circ}$-free and $\phi \Rightarrow \psi$ is provable, by the induction hypothesis $\Epd S_{\psi}^{a}, \phi \Rightarrow \Ap S_{\psi}$. However, since $S_{\psi}^{a}=\varnothing$, we have $\Epd S_{\psi}^{a}=\top$. Therefore, $\phi \Rightarrow \Ap S_{\psi}$, and by $(M)$ and then $(Lw)$, we get $\Epd S^{a}, \bar{C} \Rightarrow \Apm S$. As $\Apm S$ is one of the disjuncts in the definition of $\Ap S$, we get $\Epd S^{a}, \bar{C} \Rightarrow \Ap S$.

If the last rule in the proof is the modal rule $(LM\!\!\to)$, then it is of the form
\begin{center}
\begin{tabular}{c}
 \AxiomC{$\phi \Rightarrow \psi$} 
 \AxiomC{$\Gamma, \Box \phi, \theta \Rightarrow \Delta$} 
   \RightLabel{\small$LM \!\!\to$} 
 \BinaryInfC{$\Gamma, \Box \phi, \Box \psi \to \theta \Rightarrow \Delta$}
 \DisplayProof
\end{tabular}
\end{center}
There are four cases to consider based on which formulas 
are in $\bar{C}$.

If $\Box \phi, \Box \psi \to \theta \in \bar{C}$, then $\phi$ and $\theta$ are $p^{\circ}$-free. Set $\bar{C'}=\bar{C}-\{\Box \phi, \Box \psi \to \theta\}$. The right premise is of the form $S \cdot (\bar{C'}, \Box \phi, \theta \Rightarrow \,)$. Therefore, by the induction hypothesis $\Epd S^{a}, \bar{C'},\Box \phi, \theta \Rightarrow \Ap S$. Applying $(LM \!\!\to)$ on the latter sequent and $\phi \Rightarrow \psi$, we get $(\Epd S^{a}, \bar{C'}, \Box \phi, \Box \psi \to \theta \Rightarrow \Ap S)=(\Epd S^{a}, \bar{C} \Rightarrow \Ap S)$.


Note that in the other three cases below, we have $S^{a}\neq \varnothing$. Hence, $S$ is not of the form $(\Rightarrow \Box \psi)$ and hence $\Apm S$ is in the form of the big disjunction. 

Suppose $\Box \phi, \Box \psi \to \theta \notin \bar{C}$. 
As the right premise is of the form $S_{\Box \psi, \theta}\cdot (\bar{C} \Rightarrow)$ and by remark $(R_1)$, we have $S_{\Box \psi, \theta} \prec S$, by the induction hypothesis we have $\Epd S_{\Box \psi, \theta}^{a}, \bar{C} \Rightarrow \Ap S_{\Box \psi, \theta}$. Since both $\Box \phi$ and $\Box \psi \to \theta$ are in $S^{a}$ and $\phi \Rightarrow \psi$ is provable, we get $(\phi, \psi, \theta) \in I_m(S^{a})$, which implies that $\Epd S^{a}_{\Box \psi, \theta}$ is a conjunct in $\Epd S^{a}$ and $\Ap S_{\Box \psi, \theta}$ a disjunct in $\Ap S$. Therefore, we get $\Epd S^{a}, \bar{C} \Rightarrow \Ap S$.

If $\Box \psi \to \theta \in \bar{C}$ and $\Box \phi \notin \bar{C}$, then $\psi$ is $p^{\diamond}$-free and $\theta$ is  $p^{\circ}$-free. Set $\bar{C}'=\bar{C}-\{\Box \psi \to \theta\}$.
Since $\phi \Rightarrow \psi$ is provable and $\psi$ is $p^{\diamond}$-free, by the induction hypothesis condition \eqref{exists2} we have $\Epd \phi \Rightarrow \psi$. As $S \cdot (\bar{C'}, \theta \Rightarrow)$ is a premise of the rule and has a shorter proof, by the induction hypothesis we have $\Epd S^{a}, \bar{C'}, \theta \Rightarrow \Ap S$ and by $(Lw)$, $\Box \Epd \phi, \Epd S^{a}, \bar{C'}, \theta \Rightarrow \Ap S$. Applying $(LM \!\!\to)$ on the latter sequent and on $\Epd \phi \Rightarrow \psi$ we get
    $\Box \Epd \phi, \Epd S^{a}, \bar{C'}, \Box \psi \to \theta \Rightarrow \Ap S$.
As $\Box \phi \in S^a$, by definition 
$\Box \Epd \phi$ appears as a conjunct in $\Epd S^{a}$. Hence, $\Epd S^{a}, \bar{C} \Rightarrow \Ap S$.

If $\Box \phi \in \bar{C}$ and $\Box \psi \to \theta \notin \bar{C}$, set $\bar{C'}=\bar{C}-\{\Box \phi\}$. 
As $S_{\psi} \cdot (\phi \Rightarrow)=(\phi \Rightarrow \psi)$ is provable and $\phi$ is $p^{\circ}$-free, by the induction hypothesis $\Epd S_{\psi}^{a}, \phi \Rightarrow \Ap S_{\psi}$, or equivalently $\phi \Rightarrow \Ap S_{\psi}$, as $S^a_{\psi}=\varnothing$ and hence $\Epd S_{\psi}^{a} = \top$. Therefore, by $(M)$, 
\begin{center}
  $\Box\phi \Rightarrow \Box\Ap S_{\psi} \qquad (1)$   
\end{center} 
On the other hand, as the right premise is 
$S_{\Box \psi, \theta} \cdot (\bar{C'}, \Box \phi \Rightarrow)$  and $\bar{C'}$ and $\Box \phi$ are $p^{\circ}$-free, by the induction hypothesis, condition \eqref{forall2}, we get
\begin{center}
$\Epd S_{\Box \psi, \theta}^{a}, \bar{C'}, \Box \phi \Rightarrow \Ap S_{\Box \psi, \theta}$. \qquad (2)
\end{center}
Therefore, using $(1)$ and $(2)$ we get
\begin{center}
$\Box \Ap S_{\psi} \to \Epd S_{\Box \psi, \theta}^{a}, \bar{C'}, \Box \phi \Rightarrow \Ap S_{\Box \psi, \theta}$.
\end{center}
As $\Box \psi \to \theta \in S^{a}$, the formula $\Box \Ap S_{\psi} \to \Epd S_{\Box \psi, \theta}^{a}$ appears as a conjunct in the definition of $\Epd S^{a}$. Hence, $\Epd S^{a}, \bar{C'}, \Box \phi \Rightarrow \Ap S_{\Box \psi, \theta}$.
Together with $(1)$ we get $\Epd S^{a}, \bar{C'}, \Box \phi \Rightarrow \Ap S_{\Box \psi, \theta} \wedge\Box \Ap S_{\psi}$,
which implies $\Epd S^{a}, \bar{C} \Rightarrow \Ap S$, as $\Ap S_{\Box \psi, \theta} \wedge\Box \Ap S_{\psi}$ is a disjunct in $\forall^{\circ} p S$, again by $\Box \psi \to \theta \in S^{a}$.
\end{proof}

\subsection*{Acknowledgements}
We thank three anonymous referees for valuable comments on an earlier version of this paper, and one in particular for the careful reading of the manuscript and the detailed suggestions for improvement. 

\bibliographystyle{aiml22}
\bibliography{aiml22}
\end{document}